\documentclass[a4paper,reqno,10pt]{amsart}

\usepackage{epsf,graphicx,epsfig}
\usepackage{amscd}
\usepackage{amsmath,latexsym,amssymb,amsthm}
\usepackage[nospace,noadjust]{cite}
\usepackage{textcomp}
\usepackage{setspace,cite}
\usepackage{lscape,fancyhdr,fancybox}
\usepackage{stmaryrd}
\usepackage[all,cmtip]{xy}
\usepackage{tikz}
\usepackage{cancel}
\usetikzlibrary{shapes,arrows,decorations.markings}
\usepackage{graphicx}

\usepackage{color}
\usepackage{url}
\usepackage{enumerate}
\usepackage[mathscr]{euscript}

  \theoremstyle{plain}

\swapnumbers
    \newtheorem{thm}{Theorem}[section]
    \newtheorem{prop}[thm]{Proposition}

    \newtheorem{corollary}[thm]{Corollary}
    
    \newtheorem{subsec}[thm]{}
\theoremstyle{definition}
    \newtheorem{defn}[thm]{Definition}
        \newtheorem{remark}[thm]{Remark}
    \newtheorem{exam}[thm]{Example}

\theoremstyle{remark}

\title{}
\author{}
\date{}
\usepackage{amssymb}

\usepackage{hyperref}
\hypersetup{
	colorlinks,
	citecolor=blue,
	filecolor=black,
	linkcolor=blue,
	urlcolor=black
}

\begin{document}

\title{Pre-Leibniz algebras}

\author{Apurba Das}
\address{Department of Mathematics,
Indian Institute of Technology, Kharagpur 721302, West Bengal, India.}
\email{apurbadas348@gmail.com}

\begin{abstract}
The notion of pre-Leibniz algebras was recently introduced in the study of Rota-Baxter operators on Leibniz algebras. In this paper, we first construct a graded Lie algebra whose Maurer-Cartan elements are pre-Leibniz algebras. Using this characterization, we define the cohomology of a pre-Leibniz algebra with coefficients in a representation. This cohomology is shown to split the Loday-Pirashvili cohomology of Leibniz algebras. As applications of our cohomology, we study formal and finite order deformations of a pre-Leibniz algebra. Finally, we define homotopy pre-Leibniz algebras and classify some special types of homotopy pre-Leibniz algebras.
\end{abstract}

\maketitle



\medskip



\medskip


\thispagestyle{empty}

\tableofcontents


\medskip


\section{Introduction}\label{sec1}

The notion of pre-Lie algebras was first introduced by Gerstenhaber in his work on deformations of algebras \cite{gers,gers2}. A pre-Lie algebra naturally induces a Lie algebra structure on the underlying vector space.
Pre-Lie algebras have found important connections with rooted trees, vector fields on affine spaces, deformation theory and Rota-Baxter operators on Lie algebras \cite{livernet,diatta,dotsenko,kuper}. Cohomology and deformation theory for pre-Lie algebras was introduced by Dzhumadil'daev \cite{dz}. Such study generalizes the classical Hochschild cohomology theory and deformation theory for associative algebras \cite{gers,hoch}. The notion of homotopy pre-Lie algebras was described by Chapoton and Livernet in terms of rooted trees \cite{livernet}. For more on pre-Lie algebras, see \cite{burde,bai4,bai6,segal50}.

\medskip

On the other hand, Leibniz algebras were first considered by Bloh \cite{bloh,bloh1} and rediscovered by Loday \cite{loday-une} in the study of periodicity phenomenon in algebraic $K$-theory. Leibniz algebras are roughly a noncommutative generalization of Lie algebras. In \cite{loday-pira} Loday and Pirashvili introduced a cohomology theory for Leibniz algebras with coefficients in a representation. Given a vector space $\mathfrak{g}$, Balavoine \cite{bala-operad} constructs a graded Lie algebra (known as Balavoine's graded Lie algebra) on the space of multilinear maps on $\mathfrak{g}$ whose Maurer-Cartan elements correspond to Leibniz algebra structures on $\mathfrak{g}$. 
Recently, Rota-Baxter operators and relative Rota-Baxter operators on Leibniz algebras are introduced and their relation with Leibniz Yang-Baxter equation and Leibniz bialgebras are discovered in \cite{ra,sheng-pl,sheng-defr}. Like a Rota-Baxter operator on a Lie algebra induces a pre-Lie algebra structure, a (relative) Rota-Baxter operator on a Leibniz algebra gives rise to a pre-Leibniz algebra (already introduced in \cite{sheng-pl} by the name of Leibniz-dendriform algebra) structure. A pre-Leibniz algebra $(\mathfrak{a}, \triangleleft, \triangleright)$ is a vector space $\mathfrak{a}$ together with two bilinear operations $ \triangleleft , \triangleright : \mathfrak{a} \otimes \mathfrak{a} \rightarrow \mathfrak{a}$ satisfying a set of three identities (see Definition \ref{pl-defn}). It turns out that the sum of the two operations is indeed a Leibniz algebra structure on $\mathfrak{a}$, called the total Leibniz algebra and denoted by $\mathfrak{a}_\mathrm{Tot}$. Thus, pre-Leibniz  algebras can be considered as a splitting of Leibniz algebras. Nothing is known about this newly defined pre-Leibniz algebra structure. Our aim in this paper is to give an extensive study of this new structure. 

\medskip

Given a vector space $\mathfrak{a}$, we first construct a graded Lie algebra whose Maurer-Cartan elements are in one-to-one correspondence with pre-Leibniz algebra structures on $\mathfrak{a}$. This graded Lie algebra can be considered as a splitting of the Balavoine's graded Lie algebra. Next, we introduce representations of a pre-Leibniz algebra and define the cohomology of a pre-Leibniz algebra with coefficients in a representation. We show that there is a morphism from the cohomology of a pre-Leibniz algebra $(\mathfrak{a}, \triangleleft, \triangleright )$ to the Loday-Pirashvili cohomology of the total Leibniz algebra $\mathfrak{a}_\mathrm{Tot}$ (cf. Theorem \ref{rel-lp-coho}). 

\medskip

Next, we consider formal, infinitesimal and finite order deformations of a pre-Leibniz algebra $(\mathfrak{a}, \triangleleft, \triangleright)$. We show that there is a one-to-one correspondence between the equivalence classes of infinitesimal deformations of $(\mathfrak{a}, \triangleleft, \triangleright)$ and the second cohomology group with coefficients in the adjoint representation (cf. Theorem \ref{inf-def-thm} and Remark \ref{inf-def-rmk}). A pre-Leibniz algebra is said to be rigid if any formal deformation is trivial. We show that the vanishing of the second cohomology group implies that the pre-Leibniz algebra is rigid (cf. Theorem \ref{rigid-thm}, Remark \ref{rigid-rmk}). Finally, given a finite order deformation of a pre-Leibniz algebra, we associate a third cohomology class (called the obstruction class) which is the obstruction to extend the given deformation to the next order (cf. Theorem \ref{obs-thm}).

\medskip

The notion of Leibniz$_\infty$-algebras (homotopy Leibniz algebras) was first introduced by Ammar and Poncin \cite{ammar-poncin} as the homotopy analogue of Leibniz algebras. See \cite{chen-sti-xu,khuda-poncin-qiu,uchi} for more on Leibniz$_\infty$-algebras. In this paper, we introduce pre-Leibniz$_\infty$-algebras as the homotopy version of pre-Leibniz algebras. We show that a pre-Leibniz$_\infty$-algebra induces a Leibniz$_\infty$-algebra structure on the underlying graded vector space (cf. Theorem \ref{pl-l-inf}). This generalizes the construction of the total Leibniz algebra associated with a pre-Leibniz algebra. We introduce Rota-Baxter operators on Leibniz$_\infty$-algebras that gives rise to a pre-Leibniz$_\infty$-algebras (cf. Theorem \ref{rota-pl-inf}). Next, we focus on pre-Leibniz$_\infty$-algebras whose underlying graded vector space is concentrated in degree $-1$ and $0$. Some particular classes of such homotopy algebras are given by `skeletal' and `strict' pre-Leibniz$_\infty$-algebras. We show that skeletal pre-Leibniz$_\infty$-algebras are classified by third cocycles of pre-Leibniz algebras (cf. Theorem \ref{thm-skeletal}). Finally, we introduce crossed module of pre-Leibniz algebras and show that strict pre-Leibniz$_\infty$-algebras are in one-to-one correspondence with crossed module of pre-Leibniz algebras (cf. Theorem \ref{thm-strict}).

\medskip

The paper is organized as follows. In Section \ref{sec-2}, we recall Leibniz algebras, their cohomology and Balavoine's graded lie algebra. Pre-Leibniz algebras are considered and a new graded Lie algebra associated to a vector space $\mathfrak{a}$ is constructed in Section \ref{sec-3}. The Maurer-Cartan elements of this graded Lie algebra correspond to pre-Leibniz algebra structures on $\mathfrak{a}$. In Section \ref{sec-4}, we introduce representations of pre-Leibniz algebras and define cohomology with coefficients in a representation. Deformation theory of pre-Leibniz algebras is extensively studied in Section \ref{sec-5} in terms of cohomology theory. Finally, we introduce pre-Leibniz$_\infty$-algebras, their relation with Leibniz$_\infty$-algebras, and classification of skeletal and strict pre-Leibniz$_\infty$-algebras are given in Section \ref{sec-6}.

\medskip

All vector spaces, (multi)linear maps and tensor products are over a field ${\bf k}$ of characteristic $0$.








\section{Leibniz algebras and the Loday-Pirashvili cohomology}\label{sec-2}
In this preliminary section, we recall Leibniz algebras, their representations and the Loday-Pirashvili cohomology of Leibniz algebras. Our main references are \cite{bala-operad,loday-une,loday-pira}.

\begin{defn}\label{defn-leib}
A (left) Leibniz algebra is a pair $(\mathfrak{g}, [~,~])$ consists of a vector space $\mathfrak{g}$ and a bilinear bracket $[~,~] : \mathfrak{g} \otimes \mathfrak{g} \rightarrow \mathfrak{g}$ satisfying
\begin{align}\label{leib-iden}
[x, [y, z]] = [[x,y], z] + [y, [x,z]], ~\text{ for } x, y, z \in \mathfrak{g}.
\end{align}
Sometimes we denote a Leibniz algebra simply by $\mathfrak{g}$ when the bracket is clear from the context.
\end{defn}

\begin{remark}
Note that the identity (\ref{leib-iden}) is equivalent to to say that the left multiplications $[x, -] : \mathfrak{g} \rightarrow \mathfrak{g}$ are derivations for the bracket. This is the reason the Leibniz algebras of Definition \ref{defn-leib} are called left Leibniz algebras. Similarly, one has right Leibniz algebras as a pair $(\mathfrak{g}, [~,~])$ for which the right multiplications $[-,x] : \mathfrak{g} \rightarrow \mathfrak{g}$ are derivations for the bracket. It is easy to see that if $(\mathfrak{g}, [~,~])$ is a left Leibniz algebra then $(\mathfrak{g}, [~,~]^{op})$ is a right Leibniz algebra and vice versa, where $[x,y]^{op} = [y,x]$, for $x, y \in \mathfrak{g}$. Therefore, any results for left Leibniz algebras can be easily generalized to right Leibniz algebras without much hard work. Throughout the paper, by a Leibniz algebra, we shall always mean a left Leibniz algebra.
\end{remark}

\begin{defn}
Let $(\mathfrak{g}, [~,~])$ be a Leibniz algebra. A representation of it consists of a triple $(M, \rho^L, \rho^R)$ in which $M$ is a vector space and $\rho^L : \mathfrak{g} \otimes M \rightarrow M$, $\rho^R : M \otimes \mathfrak{g} \rightarrow M$ are bilinear maps (called left and right $\mathfrak{g}$-actions) satisfying for $x, y \in \mathfrak{g}$, $u \in M$, 
\begin{align}
\rho^L (x , \rho^L (y, u)) =~& \rho^L ([x,y], u) + \rho^L (y, \rho^L (x, u)), \\
\rho^L (x, \rho^R (u, y)) =~& \rho^R (\rho^L (x, u), y) + \rho^R (u, [x,y]),\\
\rho^R (u, [x,y]) =~& \rho^R (\rho^R (u,x), y) + \rho^L (x, \rho^R (u, y)). 
\end{align} 
\end{defn}

We may also denote a representation $(M, \rho^L, \rho^R)$ simply by $M$ when the left and right $\mathfrak{g}$-actions are clear.
Let $(\mathfrak{g}, [~,~])$ be a Leibniz algebra. Then $\mathfrak{g}$ is a representation of itself with both left and right $\mathfrak{g}$-actions are given by the Leibniz bracket, i.e. $\rho^L (x, y) = \rho^R (x,y) = [x,y]$, for $x, y \in \mathfrak{g}$. This is called the adjoint representation.


\medskip

Let $(\mathfrak{g}, [~,~])$ be a Leibniz algebra and $(M, \rho^L, \rho^R)$ be a representation of it. The Loday-Pirashvili cochain complex $\{ C^\ast_{\mathsf{LP}} (\mathfrak{g}, M), \delta_{\mathsf{LP}} \}$ is given by the following: the $n$-th cochain group $C^n_{\mathsf{LP}} (\mathfrak{g}, M) = \mathrm{Hom}(\mathfrak{g}^{\otimes n}, M) $ and the coboundary map $\delta_{\mathsf{LP}} : C^n_{\mathsf{LP}} (\mathfrak{g}, M) \rightarrow C^{n+1}_{\mathsf{LP}} (\mathfrak{g}, M)$ given by
\begin{align}\label{lp-coboundary}
(\delta_{\mathsf{LP}} f ) (x_1, \ldots, x_{n+1}) =~& \sum_{i=1}^n (-1)^{i+1}~ \rho^L \big( x_i, f(x_1, \ldots, \widehat{x_i}, \ldots, x_{n+1}) \big) + (-1)^{n+1} ~\rho^R \big( f (x_1, \ldots, x_n) , x_{n+1} \big) \\
&+ \sum_{1 \leq i < j \leq n+1} (-1)^i ~ f ( x_1, \ldots, \widehat{x_i}, \ldots, x_{j-1}, [x_i, x_j], x_{j+1}, \ldots, x_{n+1} ), \nonumber
\end{align}
for $x_1, \ldots, x_{n+1} \in \mathfrak{g}$. The cohomology groups of the cochain complex $\{ C^\ast_{\mathsf{LP}} (\mathfrak{g}, M), \delta_{\mathsf{LP}} \}$
are called the Loday-Pirashvili cohomology of the Leibniz algebra $\mathfrak{g}$ with coefficients in the representation $M$, and they are denoted by $H^\ast_{\mathsf{LP}} (\mathfrak{g}, M)$.

Next we recall the graded Lie algebra which characterize Leibniz algebras as Maurer-Cartan elements \cite{bala-operad}. Let $\mathfrak{g}$ be a vector space, not necessarily a Leibniz algebra. Then the graded space $C^\ast_{\mathsf{LP}} (\mathfrak{g}, \mathfrak{g}) = \oplus_{n \geq 1} C^n_{\mathsf{LP}} (\mathfrak{g}, \mathfrak{g}) =  \oplus_{n \geq 1} \mathrm{Hom} ( \mathfrak{g}^{\otimes n }, \mathfrak{g})$ of multilinear maps on $\mathfrak{g}$ carries a degree $-1$ graded Lie bracket (called the Balavoine bracket)
\begin{align}
\llbracket f, g \rrbracket_\mathsf{B} := \sum_{i =1}^m (-1)^{(i-1)(n-1)}~ f \circ_i g  ~-~(-1)^{(m-1)(n-1)} \sum_{i =1}^n (-1)^{(i-1)(m-1)}~ g \circ_i f,
\end{align}
for $f \in C^m_{\mathsf{LP}} (\mathfrak{g}, \mathfrak{g})$ and $g \in C^n_{\mathsf{LP}} (\mathfrak{g}, \mathfrak{g})$, where
\begin{align}
(f \circ_i g ) & ( x_1, \ldots, x_{m+n-1}) \\
&=\sum_{\sigma \in Sh (i-1,n-1)} (-1)^\sigma ~ f (x_{\sigma (1)}, \ldots, x_{\sigma (i-1)}, g (x_{\sigma (i)}, \ldots, x_{\sigma (i+n-2)}, x_{i+n-1}), \ldots, x_{m+n-1}). \nonumber
\end{align}
In other words, $( C^{\ast + 1 }_{\mathsf{LP}} (\mathfrak{g}, \mathfrak{g}), \llbracket ~,~ \rrbracket_\mathsf{B} )$ is a graded Lie algebra, called the Balavoine's graded Lie algebra.

Let $(\mathfrak{g}, [~,~])$ be a Leibniz algebra. Let $\mu \in C^2_{\mathsf{LP}} (\mathfrak{g}, \mathfrak{g})$ denotes the element corresponding to the bracket $[~,~]$, i.e., $\mu (x, y) = [x, y]$, for $x, y \in \mathfrak{g}$. Then the Leibniz identity (\ref{leib-iden}) is equivalent to $\llbracket \mu, \mu \rrbracket_\mathsf{B} = 0$. In fact, we have the following.

\begin{prop}
Let $\mathfrak{g}$ be a vector space. Then there is a one-to-one correspondence between Leibniz algebra structures on $\mathfrak{g}$ and Maurer-Cartan elements in the graded Lie algebra $( C^{\ast + 1 }_{\mathsf{LP}} (\mathfrak{g}, \mathfrak{g}), \llbracket ~,~ \rrbracket_\mathsf{B} )$.
\end{prop}

With the above notation, the coboundary map (\ref{lp-coboundary}) of the Leibniz algebra $\mathfrak{g}$ with coefficients in itself is given by
$\delta_{\mathsf{LP}} f = (-1)^{n-1} \llbracket \mu, f \rrbracket_\mathsf{B}, ~ \text{ for } f \in C^n_{\mathsf{LP}} (\mathfrak{g}, \mathfrak{g}).$

\section{Pre-Leibniz algebras and their Maurer-Cartan characterizations}\label{sec-3}
The notion of a pre-Leibniz algebra first appeared in the recent article \cite{sheng-pl} where it was named as Leibniz-dendriform algebra. In this section, we first {recall pre-Leibniz algebras} and their relation with relative Rota-Baxter operators. Next, we construct a new graded Lie algebra associated with a vector space, whose Maurer-Cartan elements are precisely pre-Leibniz algebra structures on that vector space.


\begin{defn} \label{pl-defn} \cite{sheng-pl}
A (left) pre-Leibniz algebra is a triple $(\mathfrak{a}, \triangleleft, \triangleright)$ consisting of a vector space $\mathfrak{a}$ together with two bilinear operations $\triangleleft, \triangleright : \mathfrak{a} \otimes \mathfrak{a} \rightarrow \mathfrak{a}$ satisfying the following identities 
\begin{align}
x \triangleleft ( y \triangleleft z + y \triangleright z ) =~& ( x \triangleleft y ) \triangleleft z + y \triangleright (x \triangleleft z ), \label{pl-iden1}\\
x \triangleright ( y \triangleleft z ) =~& ( x \triangleright y ) \triangleleft z + y \triangleleft ( x \triangleleft z + x \triangleright z),\\
x \triangleright ( y \triangleright z) =~& ( x  \triangleleft y + x \triangleright y) \triangleright z + y \triangleright ( x \triangleright z), ~\text{ for } x, y, z \in \mathfrak{a}. \label{pl-iden3}
\end{align}
\end{defn}

\begin{remark}
Let $(\mathfrak{a}, \triangleleft, \triangleright)$ be a pre-Leibniz algebra in which $x\triangleleft y = - y \triangleright x$, for $x, y \in \mathfrak{a}$ (we call such pre-Leibniz algebras `skew-symmetric'). In this case, all the three identities (\ref{pl-iden1})-(\ref{pl-iden3}) are equivalent to
\begin{align*}
( x \triangleleft y ) \triangleleft z - x \triangleleft (y \triangleleft z) = (x \triangleleft z) \triangleleft y - x \triangleleft (z \triangleleft y), ~ \text{ for } x, y, z \in \mathfrak{a}.
\end{align*} 
In other words, $(\mathfrak{a}, \triangleleft)$ is a pre-Lie algebra \cite{gers2}. Therefore, skew-symmetric pre-Leibniz algebras are equivalent to pre-Lie algebras. In this regard, one can think pre-Leibniz algebras as the noncommutative analogue of pre-Lie algebras.
\end{remark}

\begin{defn}
Let $(\mathfrak{a}, \triangleleft, \triangleright)$ and $(\mathfrak{a}', \triangleleft', \triangleright')$ be two pre-Leibniz algebras. A morphism of pre-Leibniz algebras from $(\mathfrak{a}, \triangleleft, \triangleright)$ to $(\mathfrak{a}', \triangleleft', \triangleright')$ is a linear map $\phi : \mathfrak{a} \rightarrow \mathfrak{a}'$ satisfying $\phi (x \triangleleft y) = \phi (x) \triangleleft' \phi (y)$ and $\phi (x \triangleright y) = \phi(x) \triangleright' \phi (y)$, for $x, y \in \mathfrak{a}.$
\end{defn}

Pre-Leibniz algebras are splitting of Leibniz algebras in the following sense \cite{sheng-pl}.

\begin{prop}\label{pre-leib-sum} Let $(\mathfrak{a}, \triangleleft, \triangleright)$ be a pre-Leibniz algebra. Then $(\mathfrak{a}, [~,~]_{\triangleleft, \triangleright})$ is a Leibniz algebra, where 
\begin{align*}
[x,y]_{\triangleleft, \triangleright}:= x \triangleleft y + x \triangleright y, \text{ for } x, y \in \mathfrak{a}.
\end{align*}
We call this the total Leibniz algebra and denoted by $\mathfrak{a}_\mathrm{Tot}$.
\end{prop}

A class of pre-Leibniz algebras arise from (relative) Rota-Baxter operators in the Leibniz algebra context.

\begin{defn}\label{defi-rel-rota}
Let $(\mathfrak{g}, [~,~])$ be a Leibniz algebra and $(M, \rho^L, \rho^R)$ be a representation of it.
A relative Rota-Baxter operator on $M$ over the Leibniz algebra $\mathfrak{g}$ is a linear map $T : M \rightarrow \mathfrak{g}$ satisfying
\begin{align*}
[Tu, Tv] = T ( \rho^R ( u, Tv) + \rho^L (Tu, v) ), ~\text{ for } u, v \in M.
\end{align*}
A Rota-Baxter operator on a Leibniz algebra $\mathfrak{g}$ is a relative Rota-Baxter operator on $\mathfrak{g}$ (considered as the adjoint representation space) over the Leibniz algebra $\mathfrak{g}$.
\end{defn}

\begin{prop}
Let $T : M \rightarrow \mathfrak{g}$ be a relative Rota-Baxter operator on $M$ over the Leibniz algebra $\mathfrak{g}$. Then $(M, \triangleleft, \triangleright)$ is a pre-Leibniz algebra, where
\begin{align*}
u \triangleleft v : = \rho^R ( u, Tv) ~~~ \text{ and } ~~~ u \triangleright v := \rho^L (Tu, v), ~ \text{ for } u, v \in M.
\end{align*}
Therefore, $(M, [~,~]_T)$ is a Leibniz algebra, where $[u, v]_T := \rho^R ( u, Tv) + \rho^L (Tu, v)$, for $u, v \in M$.
\end{prop}

In the previous proposition, we show that a relative Rota-Baxter operator on the representation space over a Leibniz algebra induces a pre-Leibniz algebra structure. We will now consider the converse. Let $(\mathfrak{a}, \triangleleft, \triangleright)$ be a pre-Leibniz algebra. Take the {total Leibniz algebra} $\mathfrak{a}_\mathrm{Tot} = (\mathfrak{a}, [~,~]_{\triangleleft, \triangleright})$ given in Proposition \ref{pre-leib-sum}. It can be checked that the total Leibniz algebra $\mathfrak{a}_\mathrm{Tot}$ has a representation on the vector space $\mathfrak{a}$ with left and right $\mathfrak{a}_\mathrm{Tot}$-actions
\begin{align*}
\varrho^L (u, v) = u \triangleright v ~~~ \text{ and } ~~~ \varrho^R (v, u) = v \triangleleft u, ~ \text{for } u \in \mathfrak{a}_\mathrm{Tot}, v \in \mathfrak{a}.
\end{align*}
With this notation, the identity map $\mathrm{id} : \mathfrak{a} \rightarrow \mathfrak{a}_\mathrm{Tot}$ is a relative Rota-Baxter operator on $\mathfrak{a}$ over the total Leibniz algebra $\mathfrak{a}_\mathrm{Tot}$. Moreover, the induced pre-Leibniz algebra structure on $\mathfrak{a}$ coincides with the given one.


\medskip

\noindent {\bf A new graded Lie algebra.} Given a vector space, we construct a graded Lie algebra whose Maurer-Cartan elements correspond to pre-Leibniz structures on the vector space. 

Let $C_n$ be the set of first $n$ natural numbers. For some convenience, we will denote the elements of $C_n$ as certain symbols, and write $C_n = \{ [1], [2], \ldots, [n] \}$, for $n \geq 1$. Let $m , n \geq 1$ and $1 \leq i \leq m$. For any $\sigma \in Sh (i-1, n-1)$, we put the elements of $C_{m+n-1}$ into $m$ many boxes as follows

\begin{align*}
\framebox{$[\sigma(1)]$} \quad \framebox{$[\sigma(2)]$} \quad \cdots \quad \framebox{$[\sigma(i-1)]$} \quad \framebox{$[\sigma(i)]$\quad$\cdots$ $[\sigma(i+n-2)]$ ~ [i+n-1]} \quad \framebox{[i+n]} \quad \cdots \quad \framebox{[m+n-1]}~.
\end{align*}

\medskip

\medskip

\noindent Note that, we put exactly one element in each of the first $(i-1)$ boxes, put $n$ elements in the $i$-th box, and exactly one element in each of the last $(m-i)$ many boxes. With the above notations, we will now define maps
\begin{align}\label{comb-maps}
R^\sigma_{m; i , n} : C_{m+n-1} \rightarrow C_m ~~~~ \text{ and } ~~~~ S^\sigma_{m; i , n} : C_{m+n-1} \rightarrow {\bf k}[C_n]
\end{align}
as follows:
\begin{align*}
R^\sigma_{m; i , n}  ([r]) =~& [k] \in C_m  \qquad \text{ if } [r] \text{ appears in the }k\text{-th box},\\
S^\sigma_{m; i , n}  ([r]) =~& \begin{cases} [1] + \cdots + [n] \qquad \text{ if } [r] \text{ does not appears in the }i\text{-th box}, \\
[j]  \qquad \text{ if } [r] \text{ appears in the }i\text{-th box and in the }j\text{-th position}. 
\end{cases}
\end{align*}



\medskip

Let $\mathfrak{a}$ be a vector space. For any $n \geq 1$, we consider the space 
\begin{align*}
C^n_\mathsf{pL} ( \mathfrak{a}, \mathfrak{a}) := \mathrm{Hom}( \mathbf{k}[C_n] \otimes \mathfrak{a}^{\otimes n}, \mathfrak{a}).
\end{align*}
For any $f \in C^m_\mathsf{pL} ( \mathfrak{a}, \mathfrak{a})$, $g \in C^n_\mathsf{pL} ( \mathfrak{a}, \mathfrak{a})$ and $1 \leq i \leq m$, we define an element $f \diamond_i g \in C^{m+n-1}_\mathsf{pL} ( \mathfrak{a}, \mathfrak{a})$ by
\begin{align*}
&(f \diamond_i g ) ([r]; x_1, \ldots, x_{m+n-1}) := \\
&\sum_{\sigma \in Sh (i-1,n-1)} (-1)^\sigma ~ f ( R^\sigma_{m;i,n} [r]; x_{\sigma (1)}, \ldots, x_{\sigma (i-1)}, g (S^\sigma_{m;i,n} [r]; x_{\sigma (i)}, \ldots, x_{\sigma (i+n-2)}, x_{i+n-1}), \ldots, x_{m+n-1}),
\end{align*}
for $[r] \in C_{m+n-1}$ and $x_1, \ldots, x_{m+n-1} \in \mathfrak{a}$.  We define a bracket
\begin{align*}
\llbracket ~,~ \rrbracket_\mathsf{pL} : C^m_\mathsf{pL} ( \mathfrak{a}, \mathfrak{a}) \otimes C^n_\mathsf{pL} ( \mathfrak{a}, \mathfrak{a}) \rightarrow C^{m+n-1}_\mathsf{pL} ( \mathfrak{a}, \mathfrak{a}), \text{ for } m, n \geq 1
\end{align*}
\begin{align}\label{pl-bracket}
\llbracket f, g \rrbracket_\mathsf{pL} := \sum_{i=1}^m (-1)^{(i-1)(n-1)}~ f \diamond_i g  ~- (-1)^{(m-1)(n-1)} \sum_{i=1}^n (-1)^{(i-1)(m-1)}~ g \diamond_i f.
\end{align}

\begin{thm}\label{mc-pl-char}
With the above notations, we have the followings:
\begin{itemize}
\item[(i)] $(C^\ast_{\mathsf{pL}} (\mathfrak{a}, \mathfrak{a}), \llbracket ~, ~ \rrbracket )$ is a degree $-1$ graded Lie algebra. In other words, $(C^{\ast + 1}_{\mathsf{pL}} (\mathfrak{a}, \mathfrak{a}), \llbracket ~, ~ \rrbracket )$ is a graded Lie algebra.
\item[(ii)] There is a one-to-one correspondence between pre-Leibniz algebra structures on $\mathfrak{a}$ and Maurer-Cartan elements in the graded Lie algebra $(C^{\ast + 1}_{\mathsf{pL}} (\mathfrak{a}, \mathfrak{a}), \llbracket ~, ~ \rrbracket )$.
\end{itemize}
\end{thm}

\begin{proof}
(i) Consider the double vector space $D (\mathfrak{a}) = \mathfrak{a} \oplus \mathfrak{a}$. Then any $f \in C^m_\mathsf{pL} (\mathfrak{a}, \mathfrak{a})$ gives rise to an element $D(f) \in C^m_\mathsf{LP} (D(\mathfrak{a}), D(\mathfrak{a}))$ given by
\begin{align*}
D(f) \big( (x_1,0), \ldots, (x_m, 0)  \big) =~& \sum_{r=1}^m \big(   f ([r]; x_1, \ldots, x_m), 0 \big), \\
D(f) \big( (x_1, 0), \ldots, (0, x_r), \ldots, (x_m, 0) \big) =~& \big( 0, f ([r]; x_1, \ldots, x_m) \big), \\
D(f) \big(  \ldots, (0, x_r), \ldots, (0, x_s), \ldots  \big) =~& 0.
\end{align*}
(It is easy to see that $D(f) = 0$ implies that $f=0$). If $f \in C^m_\mathsf{pL} (\mathfrak{a}, \mathfrak{a})$, $g \in C^n_\mathsf{pL} (\mathfrak{a}, \mathfrak{a})$ and $1 \leq i \leq m$, then {it is straightforward to verify} that $D(f \diamond_i g) = D(f) \circ_i D(g)$. As a consequence, we get that ${D \llbracket f, g \rrbracket_\mathsf{pL} = \llbracket D(f) , D(g) \rrbracket_\mathsf{B}}.$ This shows that $(C^\ast_\mathsf{pL} (\mathfrak{a}, \mathfrak{a}), \llbracket ~, ~ \rrbracket_\mathsf{pL})$ is a degree $-1$ graded Lie algebra as $(C^\ast_\mathsf{LP} (D(\mathfrak{a}), D(\mathfrak{a})), \llbracket ~ , ~ \rrbracket_\mathsf{LP})$ is so.

\medskip

(ii) Suppose there are bilinear maps $\triangleleft, \triangleright : \mathfrak{a} \otimes \mathfrak{a} \rightarrow \mathfrak{a}$. They corresponds to an element $\pi \in C^2_\mathsf{pL}(\mathfrak{a}, \mathfrak{a})$ given by
\begin{align}\label{pre-mc-corr}
\pi ([1]; x, y) = x \triangleleft y ~~~~ \text{ and } ~~~~ \pi ([2]; x, y) = x \triangleright y, ~ \text{ for } x, y \in \mathfrak{a}.
\end{align}
By a direct calculation using the definition, we get that
\begin{align*}
\llbracket \pi, \pi \rrbracket_\mathsf{pL} ([1]; x, y, z) =~& 2 \big( (x \triangleleft y) \triangleleft z - x \triangleleft ( y \triangleleft z + y \triangleright z) + y \triangleright (x \triangleleft z)  \big),\\
\llbracket \pi, \pi \rrbracket_\mathsf{pL} ([2]; x, y, z) =~& 2 \big( (x \triangleright y) \triangleleft z - x \triangleright (x \triangleleft z) + y \triangleleft (x \triangleleft z + x \triangleright z)  \big),\\
\llbracket \pi, \pi \rrbracket_\mathsf{pL} ([3]; x, y, z) =~& 2 \big(   (x \triangleleft y + x \triangleright y) \triangleright z - x \triangleright (y \triangleright z) + y \triangleright (x \triangleright z) \big),
\end{align*}
for $x, y, z \in \mathfrak{a}$. This shows that $(\mathfrak{a}, \triangleleft, \triangleright)$ is a pre-Leibniz algebra if and only if $\llbracket \pi, \pi \rrbracket_\mathsf{pL} = 0$. Hence the result follows.
\end{proof}

Let $(\mathfrak{a}, \triangleleft, \triangleright)$ be a pre-Leibniz algebra with the corresponding Maurer-Cartan element $\pi \in C^2_\mathsf{pL} (\mathfrak{a}, \mathfrak{a})$. Then there is a differential 
\begin{align*}
d_\pi : = \llbracket \pi, - \rrbracket_\mathsf{pL} : C^n_\mathsf{pL} (\mathfrak{a}, \mathfrak{a}) \rightarrow C^{n+1}_\mathsf{pL} (\mathfrak{a}, \mathfrak{a})
\end{align*}
which makes the triple $(C^{\ast + 1}_\mathsf{pL}(\mathfrak{a}, \mathfrak{a}), \llbracket ~, ~\rrbracket_\mathsf{pL}, d_\pi)$ into a differential graded Lie algebra. This differential graded Lie algebra is useful to study Maurer-Cartan deformations of a pre-Leibniz algebra.

\begin{prop}
Let $(\mathfrak{a}, \triangleleft, \triangleright)$ be a pre-Leibniz algebra. Then for any bilinear operations $\triangleleft', \triangleright' : \mathfrak{a} \otimes \mathfrak{a} \rightarrow \mathfrak{a}$, the triple $(\mathfrak{a}, \triangleleft + \triangleleft', \triangleright + \triangleright')$ is a pre-Leibniz algebra if and only if the element $\pi' \in C^2_\mathsf{pL}(\mathfrak{a}, \mathfrak{a} )$ defined by $\pi' ([1]; x, y) = x \triangleleft' y$ and $\pi' ([2]; x, y) = x \triangleright' y$ for $x, y \in \mathfrak{a}$, is a Maurer-Cartan element in the differential graded Lie algebra $(C^{\ast + 1}_\mathsf{pL}(\mathfrak{a}, \mathfrak{a}), \llbracket ~, ~\rrbracket_\mathsf{pL}, d_\pi)$.
\end{prop}

\begin{proof}
Observe that
\begin{align*}
\llbracket \pi + \pi', \pi + \pi' \rrbracket_\mathsf{pL} =~& \llbracket \pi, \pi \rrbracket_\mathsf{pL} + \llbracket \pi, \pi' \rrbracket_\mathsf{pL} + \llbracket \pi', \pi \rrbracket_\mathsf{pL} + \llbracket \pi', \pi' \rrbracket_\mathsf{pL} \\
=~& 2 \big(  \llbracket \pi, \pi' \rrbracket_\mathsf{pL} + \frac{1}{2} \llbracket \pi', \pi' \rrbracket_\mathsf{pL}  \big) = 2 \big(   d_\pi (\pi') + \frac{1}{2} \llbracket \pi', \pi' \rrbracket_\mathsf{pL} \big).
\end{align*}
This shows that $\pi + \pi'$ (equivalently, the pair $(\triangleleft + \triangleleft', \triangleright + \triangleright')$ ) defines a pre-Leibniz algebra structure on $\mathfrak{a}$ if and only if $\pi'$ is a Maurer-Cartan element in the differential graded Lie algebra $(C^{\ast + 1}_\mathsf{pL}(\mathfrak{a}, \mathfrak{a}), \llbracket ~, ~\rrbracket_\mathsf{pL}, d_\pi)$.
\end{proof}

\section{Representations and cohomology of pre-Leibniz algebras}\label{sec-4}
In this section, we first introduce representations of a pre-Leibniz algebra and define the cohomology of a pre-Leibniz algebra with coefficients in a representation. This cohomology in a certain sense splits the Loday-Pirashvili cohomology of Leibniz algebra. We also find the relation between the cohomology of a relative Rota-Baxter operator introduced in \cite{sheng-pl,sheng-defr} and the cohomology of the induced pre-Leibniz algebra.

\medskip

\noindent {\bf Representations of pre-Leibniz algebras.} Here we introduce representations of a pre-Leibniz algebra and construct the corresponding semidirect product.
\begin{defn}
Let $( \mathfrak{a}, \triangleleft, \triangleright)$ be a pre-Leibniz algebra. A representation of it is consists of a vector space $M$ together with four bilinear operations
\begin{align}\label{rep-opr}
\triangleleft^L : \mathfrak{a} \otimes M \rightarrow M, \qquad \triangleright^L :  \mathfrak{a} \otimes M \rightarrow M, \qquad \triangleleft^R : M \otimes \mathfrak{a} \rightarrow M ~~~~ \text{ and } ~~~~ \triangleright^R : M \otimes \mathfrak{a} \rightarrow M
\end{align}
satisfying the $9$ identities where each set of $3$ identities correspond to the identities (\ref{pl-iden1})-(\ref{pl-iden3}) with exactly one of $x, y, z$ is replaced by an element of $M$ and the corresponding bilinear operation is replaced by the respective operation from (\ref{rep-opr}). A representation as above may be denoted by $(M, \triangleleft^L, \triangleright^L, \triangleleft^R, \triangleright^R)$.
\end{defn}

Any pre-Leibniz algebra $(\mathfrak{a}, \triangleleft, \triangleright)$ can be regarded as a representation of itself with $\triangleleft^L = \triangleleft^R = \triangleleft$ and $\triangleright^L = \triangleright^R = \triangleright$. This is called the adjoint representation.

\begin{prop}\label{prop-semid}
Let $(\mathfrak{a}, \triangleleft, \triangleright )$ be a pre-Leibniz algebra and $(M, \triangleleft^L, \triangleleft^R, \triangleright^L, \triangleright^R)$ be a representation of it. Then the direct sum $\mathfrak{a} \oplus M$ carries a pre-Leibniz algebra structure given by
\begin{align*}
(x, u) ~\overline{\triangleleft}~ (y, v) : =~& ( x \triangleleft y,~ x \triangleleft^L v + u \triangleleft^R y),\\
(x, u) ~\overline{\triangleright}~ (y, v) : =~& ( x \triangleright y,~ x \triangleright^L v + u \triangleright^R y).
\end{align*}
This is called the semidirect product pre-Leibniz algebra, denoted by $\mathfrak{a} \ltimes M$.
\end{prop}

\medskip

\noindent {\bf Cohomology theory.} Let $(\mathfrak{a}, \triangleleft, \triangleright )$ be a pre-Leibniz algebra and $(M, \triangleleft^L, \triangleright^L, \triangleleft^R,  \triangleright^R)$ be a representation of it. We define elements $\pi^L \in \mathrm{Hom} ({\bf k} [C_2] \otimes \mathfrak{a} \otimes M, M)$ and $\pi^R \in \mathrm{Hom} ({\bf k} [C_2] \otimes M \otimes \mathfrak{a}, M)$ by
\begin{align*}
\begin{cases}
\pi^L ([1]; x, u) = x \triangleleft^L u \\
\pi^L ([2]; x, u) = x \triangleright^L u,
\end{cases}
~~~~ \text{ and ~~} ~~~~
\begin{cases}
\pi^R ([1];  u,x) = u \triangleleft^R x \\
\pi^R ([2]; u,x) = u \triangleright^R x,
\end{cases}
\end{align*}
for $x \in \mathfrak{a}$ and $u \in M$. With these notations, we are now in a position to define the cohomology of the pre-Leibniz algebra $(\mathfrak{a}, \triangleleft, \triangleright)$ with coefficients in the representation $(M, \triangleleft^L, \triangleright^L, \triangleleft^R,  \triangleright^R)$.

For each $n \geq 1$, we define the $n$-th cochain group $C^n_{\mathsf{pL}} ( \mathfrak{a}, M)$ as 
\begin{align*}
C^n_\mathsf{pL} (\mathfrak{a}, M) :=  \mathrm{Hom} (\mathbf{k}[C_n] \otimes \mathfrak{a}^{\otimes n}, M),
\end{align*}
and a map $\delta_\mathsf{pL} : C^n_\mathsf{pL} (\mathfrak{a}, M) \rightarrow C^{n+1}_\mathsf{pL} (\mathfrak{a}, M)$ given by
\begin{align}\label{diff-formula}
&(\delta_\mathsf{pL} f ) ([r]; x_1, \ldots, x_{n+1}) \\
&= \sum_{i=1}^n (-1)^{i+1} ~ \pi^L \big(  R^{\sigma_i}_{2;1,n}([r]); x_i, f ( S^{\sigma_i}_{2;1,n} ([r]); x_1, \ldots, \widehat{x_i}, \ldots, x_{n+1})  \big) \nonumber \\
&+ (-1)^{n+1} ~ \pi^R \big( R^{id}_{2;n,1} ([r]); f ( S^{id}_{2;n,1} ([r]); x_1, \ldots, x_n), x_{n+1}  \big) \nonumber \\
&+ \sum_{1 \leq i < j \leq n+1} (-1)^i ~ f \big(  R^{\sigma_{i, j}}_{n;j-1, 2} ([r]); x_1, \ldots, \widehat{x_i}, \ldots, x_{j-1}, \pi \big(  S^{\sigma_{i, j}}_{n;j-1, 2} ([r]); x_i, x_j \big), x_{j+1}, \ldots, x_{n+1}  \big), \nonumber
\end{align}
for $[r] \in C_{n+1}$ and $x_1, \ldots, x_{n+1} \in \mathfrak{a}$. Note that, in the first summation, we denote by $\sigma_i \in Sh (1, n-1)$ the unique permutation on the set $\{ 1, \ldots, n \}$ such that $\sigma_i (1) = i$. Similarly, on the third summation, we denote $\sigma_{i,j} \in Sh (j-2, 1)$ the unique permutation on the set $\{ 1, \ldots, j-1 \}$ such that 
\begin{align*}
\sigma_{i,j} (1) = 1,~ \ldots,~ \sigma_{i,j} (i-1)= i-1,~ \sigma_{i,j} (i) = i+1,~ \ldots, ~\sigma_{i,j} (j-2) = j-1,~ \sigma_{i,j} (j-1) = i. 
\end{align*}
 Then we have the following.

\begin{prop}
The map $\delta_\mathsf{pL}$ satisfies $\delta_\mathsf{pL} \circ \delta_\mathsf{pL} = 0$. 
\end{prop}

\begin{proof}
Since $(\mathfrak{a}, \triangleleft, \triangleright)$ is a pre-Leibniz algebra and $(M, \triangleleft^L, \triangleright^L, \triangleleft^R, \triangleright^R)$ is a representation, it follows that $(\mathfrak{a} \oplus M, \overline{\triangleleft}, \overline{\triangleright})$ is a pre-Leibniz algebra (Proposition \ref{prop-semid}). 
Define an element $\overline{\pi} \in \mathrm{Hom} ( {\bf k}[C_2] \otimes (\mathfrak{a} \oplus M)^{\otimes 2}, \mathfrak{a} \oplus M)$ by
\begin{align*}
\overline{\pi} ([1];(x,u), (y,v)) = (x, u) ~\overline{\triangleleft}~ (y, v)  ~~~~ \text{ and } ~~~~ \overline{\pi} ([2];(x,u), (y,v)) = (x, u) ~\overline{\triangleright}~ (y, v),
\end{align*}
for $(x,u), (y, v) \in \mathfrak{a} \oplus M$. Then by Theorem \ref{mc-pl-char}, $\overline{\pi}$ is a Maurer-Cartan element in the graded Lie algebra $(C^{\ast +1}_\mathsf{pL} (\mathfrak{a} \oplus M, \mathfrak{a} \oplus M), \llbracket ~, ~ \rrbracket_\mathsf{pL})$. Hence $\overline{\pi}$ induces a cochain complex $\{ C^\ast_\mathsf{pL} (\mathfrak{a} \oplus M, \mathfrak{a} \oplus M), \overline{\delta_\mathsf{pL}} \}$, where
\begin{align*}
C^n_\mathsf{pL} (\mathfrak{a} \oplus M, \mathfrak{a} \oplus M) =~& \mathrm{Hom}( {\bf k} [C_n] \otimes (\mathfrak{a} \oplus M)^{\otimes n}, \mathfrak{a} \oplus M), \text{ for } n \geq 1   ~~~~~ \text{ and } \\
~~~ \overline{\delta_\mathsf{pL}} (F) =~& (-1)^{n-1} \llbracket \overline{\pi}, F \rrbracket_{\mathsf{pL}}, \text{ for } F \in C^n_\mathsf{pL} (\mathfrak{a} \oplus M, \mathfrak{a} \oplus M).
\end{align*}
Then it can be easily checked that the inclusions
\begin{align*}
C^n_\mathsf{pL} (\mathfrak{a}, M) = \mathrm{Hom}({\bf k} [C_n] \otimes \mathfrak{a}^{\otimes n}, M) \hookrightarrow \mathrm{Hom}( {\bf k} [C_n] \otimes (\mathfrak{a} \oplus M)^{\otimes n}, \mathfrak{a} \oplus M) = C^n_\mathsf{pL} (\mathfrak{a} \oplus M, \mathfrak{a} \oplus M)
\end{align*}
makes $\{ C^\ast_\mathsf{pL} (\mathfrak{a}, M), \delta_\mathsf{pL} \}$ into a subcomplex of $\{ C^\ast_\mathsf{pL} (\mathfrak{a} \oplus M, \mathfrak{a} \oplus M), \overline{\delta_\mathsf{pL}} \}$. This implies that $\delta_\mathsf{pL} \circ \delta_\mathsf{pL} = 0$.
\end{proof}

\medskip

It follows from the above proposition that $\{ C^\ast_\mathsf{pL} (\mathfrak{a}, M), \delta_\mathsf{pL} \} $ is a cochain complex. Let $Z^n_\mathsf{pL} (\mathfrak{a}, M) \subset C^n_\mathsf{pL} (\mathfrak{a}, M)$ denote the space of $n$-cocycles and $B^n_\mathsf{pL} (\mathfrak{a}, M) \subset C^n_\mathsf{pL} (\mathfrak{a}, M)$ denote the space of $n$-coboundaries. Then we have $B^n_\mathsf{pL} (\mathfrak{a}, M) \subset Z^n_\mathsf{pL} (\mathfrak{a}, M)$. The corresponding quotients
\begin{align*}
H^n_\mathsf{pL} (\mathfrak{a}, M) := \frac{ Z^n_\mathsf{pL} (\mathfrak{a}, M)}{B^n_\mathsf{pL} (\mathfrak{a}, M)}, \text{ for } n \geq 1
\end{align*}
are called the cohomology of the pre-Leibniz algebra $\mathfrak{a}$ with coefficients in the representation $M$.

\medskip


Let $(\mathfrak{a}, \triangleleft, \triangleright)$ be a pre-Leibniz algebra. Consider the corresponding Maurer-Cartan element $\pi \in C^2_\mathsf{pL} (\mathfrak{a}, \mathfrak{a})$ in the graded Lie algebra $(C^{\ast + 1}_\mathsf{pL}(\mathfrak{a}, \mathfrak{a}), \llbracket ~, ~ \rrbracket_\mathsf{pL})$.
It follows from the expressions of (\ref{pl-bracket}) and (\ref{diff-formula}) that the differential $\delta_\mathsf{pL} :  C^n_\mathsf{pL} ( \mathfrak{a}, \mathfrak{a}) \rightarrow C^{n+1}_\mathsf{pL} ( \mathfrak{a}, \mathfrak{a})$ of the pre-Leibniz algebra $\mathfrak{a}$ with coefficients in the adjoint representation is given by
\begin{align}\label{ad-diff}
\delta_\mathsf{pL} f = (-1)^{n-1} \llbracket \pi, f \rrbracket_\mathsf{pL}, \text{ for } f \in C^n_\mathsf{pL} ( \mathfrak{a}, \mathfrak{a}).
\end{align}
The corresponding cohomology groups are denoted by $H^\ast_\mathsf{pL} (\mathfrak{a}, \mathfrak{a}).$ It follows that the cohomology of the pre-Leibniz algebra $\mathfrak{a}$ with coefficients in the adjoint representation is induced by the Maurer-Cartan element $\pi \in C^2_\mathsf{pL} (\mathfrak{a}, \mathfrak{a})$ in the  graded Lie algebra $( C^{\ast +1}_\mathsf{pL} (\mathfrak{a}, \mathfrak{a}), \llbracket ~, ~ \rrbracket_\mathsf{pL} ).$ As a consequence, we get the following.

\begin{thm}
Let $(\mathfrak{a}, \triangleleft, \triangleright )$ be a pre-Leibniz algebra. Then the shifted graded space of cohomology $H^{\ast +1 }_\mathsf{pL} (\mathfrak{a}, \mathfrak{a})$ carries a graded Lie algebra structure.
\end{thm}


\medskip

\noindent {\bf Relation with the Loday-Pirashvili cohomology.} Let $(\mathfrak{a}, \triangleleft, \triangleright)$ be a pre-Leibniz algebra and let $(M, \triangleleft^L,  \triangleright^L, \triangleleft^R, \triangleright^R)$ be a representation of it. Then we know from Proposition \ref{pre-leib-sum} that $\mathfrak{a}_\mathrm{Tot} = ( \mathfrak{a}, [~,~]_{\triangleleft, \triangleright})$ is a Leibniz algebra. Moreover, it is easy to see that $(M, \rho^L, \rho^R)$ is a representation of the Leibniz algebra $\mathfrak{a}_\mathrm{Tot}$, where
\begin{align*}
\rho^L = \triangleleft^L + \triangleright^L \qquad \text{ and } \qquad \rho^R = \triangleleft^R + \triangleright^R.
\end{align*}
We denote this representation by $M_\mathrm{Tot}.$
Therefore, we may consider the Loday-Pirashvili cohomology of the Leibniz algebra $\mathfrak{a}_\mathrm{Tot}$ with coefficients in the representation $M_\mathrm{Tot}$. The following result finds the connection between the pre-Leibniz cohomology and the Loday-Pirashvili cohomology.

\begin{thm}\label{rel-lp-coho}
The collection $\{ \Phi_n \}_{n \geq 1}$ of maps
\begin{align*}
\Phi_n : &~C^n_\mathsf{pL} (\mathfrak{a}, M) \rightarrow C^n_\mathsf{LP} (\mathfrak{a}_\mathrm{Tot}, M_\mathrm{Tot}), ~ f \mapsto f_{\mathrm{Tot}} \\
&\text{ where }~~ f_\mathrm{Tot} (x_1, \ldots, x_n) = \sum_{r=1}^n f([r]; x_1, \ldots, x_n)
\end{align*}
give rise to a morphism from the cochain complex $\{ C^\ast_\mathsf{pL} ( \mathfrak{a}, M), \delta_\mathsf{pL} \}$ defining the pre-Leibniz cohomology to the Loday-Pirashvili cochain complex $\{ C^\ast_\mathsf{LP} ( \mathfrak{a}_\mathrm{Tot}, M_\mathrm{Tot}), \delta_\mathsf{LP} \}$ defining the Leibniz algebra cohomology. As a consequence, there is a morphism $ H^\ast_\mathsf{pL} (\mathfrak{a}, M) \rightarrow H^\ast_\mathsf{LP} (\mathfrak{a}_\mathrm{Tot}, M_\mathrm{Tot})$ of cohomologies.
\end{thm}

\begin{proof}
\underline{Case I.} (adjoint representation) Let $(M, \triangleleft^L, \triangleright^L, \triangleleft^R, \triangleright^R)$ be the adjoint representation of the pre-Leibniz algebra $(\mathfrak{a}, \triangleleft, \triangleright)$. That is, $M = \mathfrak{a}$ and $ \triangleleft^L = \triangleleft^R = \triangleleft$, $\triangleright^L = \triangleright^R = \triangleright$. We have seen in (\ref{ad-diff}) that the differential $\delta_\mathsf{pL}$ in the cochain complex $\{ C^\ast_\mathsf{pL} (\mathfrak{a}, \mathfrak{a}), \delta_\mathsf{pL} \}$ is given by
$\delta_\mathsf{pL} (f) = (-1)^{n-1} \llbracket \pi, f \rrbracket_\mathsf{pL},$ for $f \in C^n_\mathsf{pL} (\mathfrak{a}, \mathfrak{a}).$
On the other hand, we know that the differential $\delta_\mathsf{LP}$ in the cochain complex $\{ C^\ast_\mathsf{LP} (\mathfrak{a}_\mathrm{Tot}, \mathfrak{a}_\mathrm{Tot}), \delta_\mathsf{LP} \}$ is given by
\begin{align*}
\delta_\mathsf{LP} (F) = (-1)^{n-1} \llbracket \pi_\mathrm{Tot}, F \rrbracket_\mathsf{B}, ~ \text{ for } F \in C^n_\mathsf{LP} (\mathfrak{a}_\mathrm{Tot}, \mathfrak{a}_\mathrm{Tot}).
\end{align*}
Here $\pi_\mathrm{Tot} \in C^2_\mathsf{LP} (\mathfrak{a}_\mathrm{Tot}, \mathfrak{a}_\mathrm{Tot})$ is the element defined by $\pi_\mathrm{Tot} (x, y) = x \triangleleft y + x \triangleright y$, for $x, y \in \mathfrak{a}_\mathrm{Tot}$. For any $f \in C^m_\mathsf{pL} (\mathfrak{a}, \mathfrak{a})$, $g \in C^n_\mathsf{pL}(\mathfrak{a}, \mathfrak{a})$ and $1 \leq i \leq m$, it is easy to see that $\Phi_{m+n-1} (f \diamond_i g) = \Phi_m (f) \circ_i \Phi_n (g)$. As a consequence, we get that $\Phi_{m+n-1} (\llbracket f, g \rrbracket_\mathsf{pL}) = \llbracket \Phi_m (f) , \Phi_n (g) \rrbracket_\mathsf{B}$. Moreover, since $\Phi_2 (\pi) = \pi_\mathrm{Tot}$, we have
\begin{align*}
( \delta_\mathsf{LP} \circ \Phi_n )(f) = (-1)^{n-1} \llbracket \pi_\mathrm{Tot}, \Phi_n (f) \rrbracket_\mathsf{B} =~& (-1)^{n-1} \llbracket \Phi_2 (\pi), \Phi_n (f) \rrbracket_B \\
=~& (-1)^{n-1} \Phi_{n+1} (\llbracket \pi, f \rrbracket_\mathsf{pL}) = (\Phi_{n+1} \circ \delta_\mathsf{pL}) (f),
\end{align*}
for $f \in C^n_\mathsf{pL}(\mathfrak{a}, \mathfrak{a})$. This shows that the collection $\{ \Phi_n \}_{n \geq 1}$ defines a morphism of cochain complexes from $\{ C^\ast_\mathsf{pL} (\mathfrak{a}, \mathfrak{a}), \delta_\mathsf{pL} \}$ to $\{  C^\ast_\mathsf{LP} (\mathfrak{a}_\mathrm{Tot}, \mathfrak{a}_\mathrm{Tot}), \delta_\mathsf{LP}  \}.$

\medskip

\noindent \underline{Case II.} (arbitrary representation) Let $(M, \triangleleft^L, \triangleright^L, \triangleleft^R, \triangleright^R)$ be an arbitrary representation of the pre-Leibniz algebra $(\mathfrak{a}, \triangleleft, \triangleright)$. Consider the semidirect product pre-Leibniz algebra $(\mathfrak{a} \oplus M, \overline{\triangleleft}, \overline{\triangleright})$ given in Proposition \ref{prop-semid}. Let $(\mathfrak{a} \oplus M)_\mathrm{Tot} = (\mathfrak{a} \oplus M, [~,~]_{ \overline{\triangleleft}, \overline{\triangleright}  })$ be the corresponding total Leibniz algebra. By case I, the collection $\{ \Psi_n \}_{n \geq 1}$ of maps
\begin{align*}
\Psi_n : C^n_\mathsf{pL} (\mathfrak{a} \oplus M, \mathfrak{a} \oplus M) \rightarrow C^n_\mathsf{LP} ( (\mathfrak{a} \oplus M)_\mathrm{Tot}, (\mathfrak{a} \oplus M)_\mathrm{Tot} ), ~ f \mapsto f_\mathrm{Tot}
\end{align*}
forms a morphism of cochain complexes from $\{  C^\ast_\mathsf{pL} (\mathfrak{a} \oplus M, \mathfrak{a} \oplus M), \delta_\mathsf{pL} \}$ to $\{ C^\ast_\mathsf{LP} ( (\mathfrak{a} \oplus M)_\mathrm{Tot}, (\mathfrak{a} \oplus M)_\mathrm{Tot} ), \delta_\mathsf{LP} \}$. We observe that
\begin{align*}
\{ C^\ast_\mathsf{pL} (\mathfrak{a}, M), \delta_\mathsf{pL} \} &\text{ is a subcomplex of } \{  C^\ast_\mathsf{pL} (\mathfrak{a} \oplus M, \mathfrak{a} \oplus M), \delta_\mathsf{pL} \}, \text{ and } \\
 \{ C^\ast_\mathsf{LP} (\mathfrak{a}_\mathrm{Tot}, M_\mathrm{Tot}), \delta_\mathsf{LP} \} &\text{ is a subcomplex of } \{ C^\ast_\mathsf{LP} ( (\mathfrak{a} \oplus M)_\mathrm{Tot}, (\mathfrak{a} \oplus M)_\mathrm{Tot} ), \delta_\mathsf{LP} \}.
\end{align*}
Finally, the result follows as the restriction of $\Psi_\ast$ to the subcomplex $\{ C^\ast_\mathsf{pL} (\mathfrak{a}, M), \delta_\mathsf{pL} \}$ is given by $\Phi_\ast$. Hence the result follows.
\end{proof}

\medskip

\section{Deformations of pre-Leibniz algebras}\label{sec-5}
In this section, we study formal and finite order deformations of pre-Leibniz algebras following the classical theory of Gerstenhaber \cite{gers}. 

\medskip

\noindent {\bf Formal deformations.}
Let $(\mathfrak{a}, \triangleleft, \triangleright)$ be a pre-Leibniz algebra. Consider the space $\mathfrak{a}[[t]]$ of formal power series in $t$ with coefficients from $\mathfrak{a}$. Then $\mathfrak{a}[[t]]$ is a $\mathbf{k}[[t]]$-module.

\begin{defn}
A formal deformation of the pre-Leibniz algebra $(\mathfrak{a}, \triangleleft, \triangleright)$ consists of a pair $(\triangleleft_t, \triangleright_t )$ of two formal power series in $t$ with coefficients from bilinear maps on $\mathfrak{a}$,
\begin{align*}
\triangleleft_t = \triangleleft_0 + t \triangleleft_1 +~ t^2 \triangleleft_2 + \cdots ~~\text{ and }~~ \triangleright_t = \triangleright_0 + t \triangleright_1 + ~t^2 \triangleright_2 + \cdots
\end{align*}
with $\triangleleft_0 = \triangleleft$ and $\triangleright_0 = \triangleright$, such that $(\mathfrak{a}[[t]], \triangleleft_t, \triangleright_t)$ is a pre-Leibniz algebra over $\mathbf{k}[[t]].$
\end{defn}

For each $i \geq 0$, we define an element $\pi_i \in C^2_\mathsf{pL} (\mathfrak{a}, \mathfrak{a}) =  \mathrm{Hom} (\mathbf{k}[C_2] \otimes \mathfrak{a}^{\otimes 2}, \mathfrak{a})$ by
\begin{align*}
\pi_i ([1]; x, y) = x \triangleleft_i y   \quad \text{ and } \quad  \pi_i ([2]; x, y) = x \triangleright_i y, ~ \text{ for } x, y \in \mathfrak{a}.
\end{align*}
With these notations, we have $\pi_0 = \pi$. Moreover, $(\triangleleft_t, \triangleright_t)$ defines a deformation of the pre-Leibniz algebra $\mathfrak{a}$ if and only if
$\llbracket \pi_0 + t \pi_1 + t^2 \pi_2 + \cdots , \pi_0 + t \pi_1 + t^2 \pi_2 + \cdots \rrbracket_\mathsf{pL} = 0.$
This is equivalent to
\begin{align}\label{formal-def-eqn}
\sum_{i+j = n} \llbracket \pi_i, \pi_j \rrbracket_\mathsf{pL} = 0, ~ \text{ for } n =0, 1, 2, \ldots .
\end{align}
Note that (\ref{formal-def-eqn}) holds automatically for $n =0$ as $\llbracket \pi, \pi \rrbracket_\mathsf{pL} = 0$. However, for $n=1$, we get $\llbracket \pi, \pi_1 \rrbracket_\mathsf{pL} = 0$, which implies that the linear term $\pi_1 \in C^2_\mathsf{pL} (\mathfrak{a}, \mathfrak{a})$ is a $2$-cocycle in the cohomology of the pre-Leibniz algebra $\mathfrak{a}$ with coefficients in itself.

\begin{defn}
Two formal deformations $(\triangleleft_t, \triangleright_t)$ and $(\triangleleft_t', \triangleright_t')$ of a pre-Leibniz algebra $(\mathfrak{a}, \triangleleft, \triangleright)$ are said to be equivalent (we write $(\triangleleft_t, \triangleright_t) \sim (\triangleleft_t', \triangleright_t')$) if there is a formal power series
\begin{align*}
\phi_t = \phi_0 + t \phi_1 + t^2 \phi_2 + \cdots ,  \text{ with } \phi_i \in \mathrm{Hom}(\mathfrak{a}, \mathfrak{a}) \text{ and } \phi_0 = \mathrm{id}_\mathfrak{a} 
\end{align*}
such that the $\mathbf{k}[[t]]$-linear map $\phi_t : \mathfrak{a}[[t]] \rightarrow \mathfrak{a}[[t]]$ is a morphism of pre-Leibniz algebras from $(\mathfrak{a}[[t]], \triangleleft_t, \triangleright_t)$ to $(\mathfrak{a}[[t]], \triangleleft_t', \triangleright_t')$.
\end{defn}

The following result is straightforward, hence we omit the details.

\begin{thm}\label{inf-def-thm}
The linear terms corresponding to equivalent deformations of $\mathfrak{a}$ are cohomologous.
 In other words, there is a well-defined map
\begin{align}\label{eq-class}
(\text{space of formal deformations of }\mathfrak{a})/ \sim ~\rightarrow H^2_\mathsf{pL} (\mathfrak{a}, \mathfrak{a}).
\end{align}
\end{thm}

\begin{remark}\label{inf-def-rmk}
Note that the map (\ref{eq-class}) is not a bijection. However, one may consider a notion of infinitesimal deformations of $\mathfrak{a}$, and the set of equivalence classes of infinitesimal deformations are in one-to-one correspondence with $H^2_\mathsf{pL} (\mathfrak{a}, \mathfrak{a})$. An infinitesimal deformation is a formal deformation over the ring $\mathbf{k}[[t]]/(t^2)$. They are also called order $1$ deformations. More precisely, an infinitesimal deformation of a pre-Leibniz algebra $(\mathfrak{a}, \triangleleft, \triangleright)$ is a pair $(\triangleleft_t = \triangleleft + t \triangleleft_1, \triangleright_t = \triangleright + t \triangleright_1)$ such that $(\mathfrak{a}[[t]]/(t^2), \triangleleft_t, \triangleright_t)$ is a pre-Leibniz algebra over $\mathbf{k}[[t]]/ (t^2)$. One may also define equivalences between infinitesimal deformations. The map 
\begin{align*}
H^2_\mathsf{pL} (\mathfrak{a}, \mathfrak{a}) \rightarrow (\text{space of infinitesimal deformations of }\mathfrak{a})/ \sim
\end{align*}
is simply given by $[(\triangleleft_1, \triangleright_1)] \mapsto [(\triangleleft_t, \triangleright_t)]$, where $(\triangleleft_t = \triangleleft + t \triangleleft_1, ~\triangleright_t = \triangleright + t \triangleright_1)$.
\end{remark}

\begin{defn}
A formal deformation $(\triangleleft_t, \triangleright_t)$ of a pre-Leibniz algebra $(\mathfrak{a}, \triangleleft, \triangleright)$ is said to be trivial if $(\triangleleft_t, \triangleright_t)$ is equivalent to the underformed one $(\triangleleft_t' = \triangleleft,~ \triangleright_t' = \triangleright)$.
\end{defn}

The following result is standard, hence we omit the proof. See \cite[Proposition 1]{gers} in the case of associative algebras.

\begin{prop}
Let $(\mathfrak{a}, \triangleleft, \triangleright)$ be a pre-Leibniz algebra. If $H^2_\mathsf{pL} (\mathfrak{a}, \mathfrak{a}) = 0$ then any formal deformation $(\triangleleft_t, \triangleright_t)$ is equivalent to a deformation $(\triangleleft_t', \triangleright_t')$ of the form
\begin{align*}
(\triangleleft_t' = \triangleleft + t^p \triangleleft_p' + t^{p+1} \triangleleft_{p+1}' + \cdots,~ \triangleright_t' = \triangleright + t^p \triangleright_p' + t^{p+1} \triangleright_{p+1}' + \cdots),
\end{align*}
where the first non-vanishing term $\pi_p' = (\triangleleft_p', \triangleright_p')$ is a $2$-cocycle in the cohomology of $\mathfrak{a}$ but not a $2$-coboundary.
\end{prop}

As a consequence, we get the following.

\begin{thm}\label{rigid-thm}
If $H^2_\mathsf{pL} (\mathfrak{a}, \mathfrak{a}) = 0$ then any formal deformation of the pre-Leibniz algebra $\mathfrak{a}$ is trivial.
\end{thm}

\begin{remark}\label{rigid-rmk}
A pre-Leibniz algebra $\mathfrak{a}$ is said to be rigid if any formal deformation of $\mathfrak{a}$ is trivial. Thus, it follows from the above theorem that the vanishing of the second cohomology group $H^2_\mathsf{pL} (\mathfrak{a}, \mathfrak{a})$ is a sufficient condition for the rigidity of the pre-Leibniz algebra $\mathfrak{a}$.
\end{remark}

\medskip

\noindent {\bf Finite order deformations.} Here we consider finite order deformations of a pre-Leibniz algebra $\mathfrak{a}$. Given a finite order deformation of $\mathfrak{a}$, we associate a third cohomology class (called the obstruction class) in the cohomology of $\mathfrak{a}$ with coefficients in itself. When the obstruction class vanishes, the given deformation extends to a deformation of the next order.

Let $(\mathfrak{a}, \triangleleft, \triangleright)$ be a pre-Leibniz algebra. Let $N \in \mathbb{N}$ be a fixed natural number. Consider the space $\mathfrak{a}[[t]]/(t^{N+1})$ of degree $N$ polynomials in $t$ with coefficients from $\mathfrak{a}$. Then $\mathfrak{a}[[t]]/(t^{N+1})$ is a $\mathbf{k}[[t]]/(t^{N+1})$-module.

\begin{defn}
An order $N$ deformation of the pre-Leibniz algebra $(\mathfrak{a}, \triangleleft, \triangleright)$ consists of a pair $(\triangleleft_t^N, \triangleright_t^N)$ of two degree $N$ polynomials of the form
\begin{align*}
(\triangleleft_t^N = \triangleleft_0 + t \triangleleft_1 + \cdots + t^N \triangleleft_N,~ \triangleright_t^N = \triangleright_0 + t \triangleright_1 + \cdots + t^N \triangleright_N)
\end{align*}
with $\triangleleft_0 = \triangleleft$ and $\triangleright_0 = \triangleright$, such that $( \mathfrak{a}[[t]]/(t^{N+1}), \triangleleft_t^N , \triangleright_t^N  )$ is a pre-Leibniz algebra over $\mathbf{k}[[t]]/(t^{N+1})$.
\end{defn}

Using previous notations, we see that the pair $(\triangleleft_t^N, \triangleright_t^N)$ is a deformation of order $N$ if and only if $\sum_{i+j=n} \llbracket \pi_i, \pi_j \rrbracket_\mathsf{pL} = 0$, for $n=0,1, \ldots, N$. These equations can be equivalently written as
\begin{align}\label{finite-def-eq}
\delta_\mathsf{pL} (\pi_n) := - \llbracket \pi, \pi_n \rrbracket_\mathsf{pL} =  \frac{1}{2} \sum_{i+j = n; i, j \geq 1} \llbracket \pi_i, \pi_j \rrbracket_\mathsf{pL}, ~\text{ for } n =0, 1, \ldots, N.
\end{align}
Motivated by the identities in (\ref{finite-def-eq}), we define an element ${Ob}_{(\triangleleft_t^N, \triangleright_t^N)} \in C^3_\mathsf{pL} (\mathfrak{a}, \mathfrak{a})$ by
\begin{align*}
{Ob}_{(\triangleleft_t^N, \triangleright_t^N)} :=  \frac{1}{2} \sum_{i+j = N+1; i, j \geq 1} \llbracket \pi_i, \pi_j \rrbracket_\mathsf{pL}.
\end{align*}

\begin{prop}
The element ${Ob}_{(\triangleleft_t^N, \triangleright_t^N)}$ is a $3$-cocycle, i.e., ${Ob}_{(\triangleleft_t^N, \triangleright_t^N)} \in Z^3_\mathsf{pL} (\mathfrak{a}, \mathfrak{a}).$
\end{prop}

\begin{proof}
We have
\begin{align*}
&\delta_\mathsf{pL} \big(   \frac{1}{2} \sum_{i+j = N+1; i, j \geq 1} \llbracket \pi_i, \pi_j \rrbracket_\mathsf{pL} \big) \\
&= \frac{1}{2} \sum_{i+j = N+1; i, j \geq 1} \llbracket \pi, \llbracket \pi_i, \pi_j \rrbracket_\mathsf{pL} \rrbracket_\mathsf{pL} \\
&= \frac{1}{2} \sum_{i+j = N+1; i, j \geq 1} \big(  \llbracket  \llbracket \pi, \pi_i \rrbracket_\mathsf{pL}, \pi_j \rrbracket_\mathsf{pL} - \llbracket \pi_i, \llbracket \pi, \pi_j \rrbracket_\mathsf{pL} \rrbracket_\mathsf{pL}   \big) \\
&= - \frac{1}{4} \sum_{i_1 + i_2 + j = N+1; i_1, i_2, j \geq 1} \llbracket  \llbracket \pi_{i_1}, \pi_{i_2} \rrbracket_\mathsf{pL}, \pi_j \rrbracket_\mathsf{pL} + \frac{1}{4} \sum_{i+ j_1 + j_2 = N+1; i, j_1, j_2 \geq 1} \llbracket \pi_i, \llbracket \pi_{j_1}, \pi_{j_2} \rrbracket_\mathsf{pL} \rrbracket_\mathsf{pL} \\
&= - \frac{1}{2} \sum_{i+j +k = N+1; i, j, k \geq 1} \llbracket  \llbracket \pi_{i}, \pi_{j} \rrbracket_\mathsf{pL}, \pi_k \rrbracket_\mathsf{pL} = 0.
\end{align*}
Hence the result follows.
\end{proof}

It follows from the above proposition that an order $N$ deformation $(\triangleleft_t^N, \triangleright_t^N)$ of the pre-Leibniz algebra $\mathfrak{a}$ associates a cohomology class $[{Ob}_{(\triangleleft_t^N, \triangleright_t^N)}] \in H^3_\mathsf{pL} (\mathfrak{a}, \mathfrak{a})$, called the obstruction class. 

\begin{defn}
An order $N$ deformation $(\triangleleft_t^N, \triangleright_t^N)$ of the pre-Leibniz algebra $\mathfrak{a}$ is said to be extensible if there exists an element $\pi_{N+1} = (\triangleleft_{N+1}, \triangleright_{N+1}) \in C^2_\mathsf{pL} (\mathfrak{a}, \mathfrak{a})$ such that 
\begin{align*}
(\triangleleft_t^{N+1} := \triangleleft_t^N + t^{N+1}~ \triangleleft_{N+1}, ~ \triangleright_t^{N+1} := \triangleright_t^N + t^{N+1}~ \triangleright_{N+1}) ~\text{ is a deformation of order } N+1.
\end{align*}
\end{defn}

The following theorem gives a necessary and sufficient condition for the extensibility of a finite order deformation.

\begin{thm}\label{obs-thm}
An order $N$ deformation $(\triangleleft_t^N, \triangleright_t^N)$ of a pre-Leibniz algebra $\mathfrak{a}$ is extensible if and only if the corresponding obstruction class $[{Ob}_{(\triangleleft_t^N, \triangleright_t^N)}] \in H^3_\mathsf{pL} (\mathfrak{a}, \mathfrak{a})$ is trivial.
\end{thm}

\begin{proof}
Suppose the order $N$ deformation $(\triangleleft^N_t, \triangleright^N_t)$ is extensible. Let $\pi_{N+1} = (\triangleleft_{N+1}, \triangleright_{N+1})$ be an element which makes $(\triangleleft^{N+1}_t, \triangleright^{N+1}_t)$ into a deformation of order $N+1$. Then it follows that
\begin{align*}
\delta_\mathsf{pL} (\pi_{N+1}) = \frac{1}{2} \sum_{i+j = N+1; i, j \geq 1} \llbracket \pi_i, \pi_j \rrbracket_\mathsf{pL} := {Ob}_{(\triangleleft_t^N, \triangleright_t^N)}.
\end{align*}
This shows that ${Ob}_{(\triangleleft_t^N, \triangleright_t^N)}$ is a coboundary. Hence the cohomology class $[{Ob}_{(\triangleleft_t^N, \triangleright_t^N)}]$ is trivial.

Conversely, let $(\triangleleft_t^N, \triangleright_t^N)$ be an order $N$ deformation such that the cohomology class $[{Ob}_{(\triangleleft_t^N, \triangleright_t^N)}]$ is trivial. Therefore, ${Ob}_{(\triangleleft_t^N, \triangleright_t^N)}$ is a coboundary, say ${Ob}_{(\triangleleft_t^N, \triangleright_t^N)} = \delta_\mathsf{pL} (\pi_{N+1})$, where $\pi_{N+1} = ( \triangleleft_{N+1}, \triangleright_{N+1})$. Then it is easy to see that $(\triangleleft_t^{N+1} := \triangleleft_t^N + t^{N+1}~ \triangleleft_{N+1}, ~ \triangleright_t^{N+1} := \triangleright_t^N + t^{N+1}~ \triangleright_{N+1})$ is a deformation of order $N+1.$ In other words, $(\triangleleft^N_t, \triangleright^N_t)$ is extensible.
\end{proof}

\begin{corollary}
Let $\mathfrak{a}$ be a pre-Leibniz algebra with $H^3_\mathsf{pL} (\mathfrak{a}, \mathfrak{a}) = 0$. Then
\begin{itemize}
\item[(i)] any finite order deformation is extensible;
\item[(ii)] any $2$-cocycle in $Z^2_\mathsf{pL} (\mathfrak{a}, \mathfrak{a})$ is the linear term of some formal deformation of $\mathfrak{a}$.
\end{itemize}
\end{corollary}

\section{Homotopy pre-Leibniz algebras}\label{sec-6}
In this section, we first recall the notion of Leibniz$_\infty$-algebras from \cite{ammar-poncin}. Then we define pre-Leibniz$_\infty$-algebras and show that they give rise to Leibniz$_\infty$-algebras. We also define Rota-Baxter operators on Leibniz$_\infty$-algebras that induce pre-Leibniz$_\infty$-algebras. Finally, we consider skeletal and strict pre-Leibniz$_\infty$-algebras and their classification results.


\begin{defn}\cite{ammar-poncin} A Leibniz$_\infty$-algebra is a graded vector space $\mathfrak{g} = \oplus_{n \in \mathbb{Z}} \mathfrak{g}_n$
together with a collection $\{ \mu_k : \mathfrak{g}^{\otimes k} \rightarrow \mathfrak{g} ~|~ \mathrm{deg} (\mu_k ) = 2-k \}_{k \geq 1}$ of multilinear maps on $\mathfrak{g}$ satisfying the following identities
\begin{align}\label{leib-inf-iden}
\sum_{i+j = n+1} \sum_{\lambda = 1}^i & \sum_{\sigma \in Sh (\lambda -1, j-1)} \chi (\sigma) (-1)^{\lambda (j-1)} (-1)^{j (|x_{\sigma (1)}| + \cdots + |x_{\sigma (\lambda -1)}|)} \\
& \mu_i \big( x_{\sigma (1)}, \ldots, x_{\sigma (\lambda -1)}, \mu_j (x_{\sigma (\lambda)}, \ldots, x_{\sigma (\lambda + j -2)}, x_{\lambda + j -1}), x_{\lambda + j}, \ldots, x_n  \big) = 0, \nonumber
\end{align}
for $n = 1, 2, \ldots .$ A Leibniz$_\infty$-algebra as above may be denoted by $(\mathfrak{g}, \{\mu_k\}_{k \geq 1})$.
\end{defn}


\medskip

We will now define pre-Leibniz$_\infty$-algebras using the combinatorial maps defined in (\ref{comb-maps}). The definition will be justified by the subsequent results that are analogous to other homotopy type algebras (see for instance \cite{baez}).

\begin{defn}
A pre-Leibniz$_\infty$-algebra is a graded vector space $\mathfrak{a} = \oplus_{n \in \mathbb{Z}} \mathfrak{a}_n$ with a collection $\{  \pi_k : \mathbf{k}[C_k] \otimes \mathfrak{a}^{\otimes k} \rightarrow \mathfrak{a} ~|~ \mathrm{deg}(\pi_k) = 2-k \}_{k \geq 1}$ of maps satisfying
\begin{align}\label{pl-inf-iden}
\sum_{i+j = n+1} & \sum_{\lambda = 1}^i  \sum_{\sigma \in Sh (\lambda -1, j-1)} \chi (\sigma) (-1)^{\lambda (j-1)} (-1)^{j (|x_{\sigma (1)}| + \cdots + |x_{\sigma (\lambda -1)}|)} \\
& \pi_i \big( R^\sigma_{i; \lambda, j }[r]; x_{\sigma (1)}, \ldots, x_{\sigma (\lambda -1)}, \pi_j ( S^\sigma_{i; \lambda, j} [r]; x_{\sigma (\lambda)}, \ldots, x_{\sigma (\lambda + j -2)}, x_{\lambda + j -1}), \ldots, x_n  \big) = 0, \nonumber
\end{align}
for $[r] \in C_n$ and $n = 1, 2, \ldots $. A pre-Leibniz$_\infty$-algebra may be denoted by $(\mathfrak{a}, \{\pi_k \}_{k \geq 1})$.
\end{defn}




In Proposition \ref{pre-leib-sum}, we have seen that a pre-Leibniz algebra induces a Leibniz algebra. The following theorem generalizes this in the homotopy context.

\begin{thm}\label{pl-l-inf}
Let $(\mathfrak{a}, \{\pi_k \}_{k \geq 1})$ be a pre-Leibniz$_\infty$-algebra. Then $(\mathfrak{a}, \{\mu_k \}_{k \geq 1})$ is a Leibniz$_\infty$-algebra, where
\begin{align*}
\mu_k : \mathfrak{a}^{\otimes k} \rightarrow \mathfrak{a}, ~ \mu_k (x_1, \ldots, x_k)  := \sum_{r=1}^k \pi_k ([r]; x_1, \ldots, x_k), ~ \text{ for } k \geq 1.
\end{align*}
\end{thm}

\begin{proof}
Since $(\mathfrak{a}, \{ \pi_k \}_{k \geq 1})$ is a pre-Leibniz$_\infty$-algebra, it follows that (\ref{pl-inf-iden}) holds for any $n \geq 1$ and $[r] \in C_n$. For a fixed $n \geq 1$, by adding these identities for all $[r] \in C_n$, we get
\begin{align}\label{iden-00}
\sum_{i+j = n+1} &\sum_{\lambda = 1}^i  \sum_{\sigma \in Sh (\lambda -1, j-1)} \chi (\sigma) (-1)^{\lambda (j-1)} (-1)^{j (|x_{\sigma (1)}| + \cdots + |x_{\sigma (\lambda -1)}|)} \nonumber \\ 
&\bigg( \sum_{r=1}^n \pi_i \big( R^\sigma_{i; \lambda, j}[r]; x_{\sigma (1)}, \ldots, x_{\sigma (\lambda -1)}, \pi_j ( S^\sigma_{i; \lambda, j}[r]; x_{\sigma (\lambda)}, \ldots, x_{\sigma (\lambda + j -2)}, x_{\lambda + j -1}), \ldots, x_n  \big)\bigg) = 0.
\end{align}
For any fixed $i, j, \lambda$ and $\sigma \in Sh (\lambda -1, j-1)$, we can write
\begin{align*}
\sum_{r=1}^n = \sum_{r \in \{ \sigma (1), \ldots, \sigma (\lambda -1) \}}  + \sum_{r \in \{ \sigma(\lambda), \ldots, \sigma (\lambda + j -2) , \lambda + j -1 \} } + \sum_{\lambda +j}^n.
\end{align*}
We observe that
\begin{align}\label{iden-01}
& \sum_{r \in \{ \sigma (1), \ldots, \sigma (\lambda -1) \}}  \pi_i \big( R^\sigma_{i; \lambda, j}[r]; x_{\sigma (1)}, \ldots, x_{\sigma (\lambda -1)}, \pi_j ( S^\sigma_{i; \lambda, j}[r]; x_{\sigma (\lambda)}, \ldots, x_{\sigma (\lambda + j -2)}, x_{\lambda + j -1}), \ldots, x_n  \big) \nonumber \\
& = \sum_{r=1}^{\lambda-1} \pi_i \big([r];  x_{\sigma(1)}, \ldots, x_{\sigma (\lambda-1)}, \mu_j ( x_{\sigma (\lambda)}, \ldots, x_{\sigma (\lambda + j -2)}, x_{\lambda + j -1} ), x_{\lambda + j}, \ldots, x_n   \big).
\end{align}
Similarly, 
\begin{align}\label{iden-02}
&\sum_{r \in \{ \sigma(\lambda), \ldots, \sigma (\lambda + j -2) , \lambda + j -1 \} } \pi_i \big( R^\sigma_{i; \lambda, j}[r]; x_{\sigma (1)}, \ldots, x_{\sigma (\lambda -1)}, \pi_j ( S^\sigma_{i; \lambda, j}[r]; x_{\sigma (\lambda)}, \ldots, x_{\sigma (\lambda + j -2)}, x_{\lambda + j -1}), \ldots, x_n  \big) \nonumber \\
&= \pi_i \big([\lambda];  x_{\sigma(1)}, \ldots, x_{\sigma (\lambda-1)}, \mu_j ( x_{\sigma (\lambda)}, \ldots, x_{\sigma (\lambda + j -2)}, x_{\lambda + j -1} ), x_{\lambda + j}, \ldots, x_n   \big),
\end{align}
and
\begin{align}\label{iden-03}
&\sum_{\lambda +j}^n \pi_i \big( R^\sigma_{i; \lambda, j}[r]; x_{\sigma (1)}, \ldots, x_{\sigma (\lambda -1)}, \pi_j ( S^\sigma_{i; \lambda, j}[r]; x_{\sigma (\lambda)}, \ldots, x_{\sigma (\lambda + j -2)}, x_{\lambda + j -1}), \ldots, x_n  \big) \nonumber \\
&= \sum_{r = \lambda +1 }^i \pi_i \big([r];  x_{\sigma(1)}, \ldots, x_{\sigma (\lambda-1)}, \mu_j ( x_{\sigma (\lambda)}, \ldots, x_{\sigma (\lambda + j -2)}, x_{\lambda + j -1} ), x_{\lambda + j}, \ldots, x_n   \big).
\end{align}
By substituting (\ref{iden-01}), (\ref{iden-02}) and (\ref{iden-03}) in the identity (\ref{iden-00}), we get that $(\mathfrak{a}, \{ \mu_k \}_{k \geq 1})$ is a Leibniz$_\infty$-algebra.
\end{proof}

In the following, we introduce Rota-Baxter operators on Leibniz$_\infty$-algebras. Note that representations of Leibniz$_\infty$-algebras are defined in \cite{chen-sti-xu}. The notion of relative Rota-Baxter operators in the context of Leibniz$_\infty$-algebras can be defined similarly.

\begin{defn}
Let $(\mathfrak{g}, \{\mu_k \}_{k \geq 1} )$ be a Leibniz$_\infty$-algebra. A degree $0$ linear map $T : \mathfrak{g} \rightarrow \mathfrak{g}$ is said to be a Rota-Baxter operator if it satisfies
\begin{align*}
T (\mu_k (x_1, \ldots, x_k)) = \sum_{r=1}^k \mu_k (Tx_1, \ldots, Tx_{r-1}, x_r, Tx_{r+1}, \ldots, Tx_k), ~ \text{ for } k \geq 1.
\end{align*}
\end{defn}

Let $(\mathfrak{g}, [~,~])$ be a Leibniz algebra and $R: \mathfrak{g} \rightarrow \mathfrak{g}$ {be a Rota-Baxter operator on it}. In this case, the pair $(\mathfrak{g}, R)$ is often called a Rota-Baxter Leibniz algebra. A representation of a Rota-Baxter Leibniz algebra $(\mathfrak{g}, R)$ consists of a pair $(M, R_M)$ in which $M = (M, \rho^L, \rho^R)$ is a representation of the Leibniz algebra $\mathfrak{g}$ and $R_M : M \rightarrow M$ is a linear map satisfying
\begin{align*}
R_M (\rho^L (x,u)) =  R_M (\rho^L (x, R_M (u)) + \rho^L (R(x), u) ), ~~~~ R_M (\rho^R (u,x)) = R_M (\rho^R (u, R(x)) + \rho^R (R_M(u), x)), 
\end{align*}
for $ x \in \mathfrak{g}, u \in M.$ Let $(\mathfrak{g}, R)$ be a Rota-Baxter Leibniz algebra and $(M, R_M)$ be a representation of it. Then it is easy to see that $(\underbrace{M}_{-1} \oplus \underbrace{\mathfrak{g}}_0, \{\mu_k \}_{k \geq 1})$ is a Leibniz$_\infty$-algebra, where
\begin{align*}
\mu_1 = 0, ~~~~ \mu_2 (x, y) = [x,y], ~~~~ \mu_2 (x, u) = \rho^L (x, u), ~~~~ \mu_2 (u,x) = \rho^R (u, x) ~~~~ \text{and } ~ ~ \mu_k = 0 \text{ for } k \geq 3.
\end{align*}
Moreover, the map $T : M \oplus \mathfrak{g} \rightarrow M \oplus \mathfrak{g}$, $T ((u, x)) = (R_M (u), R(x))$ is a Rota-Baxter operator on the Leibniz$_\infty$-algebra $(M \oplus \mathfrak{g}, \{\mu_k \}_{k \geq 1})$.

\begin{thm}\label{rota-pl-inf}
Let $T: \mathfrak{g} \rightarrow \mathfrak{g}$ be a Rota-Baxter operator on the Leibniz$_\infty$-algebra $(\mathfrak{g}, \{\mu_k \}_{k \geq 1} )$. Then $(\mathfrak{g}, \{\pi_k \}_{k \geq 1})$ is a pre-Leibniz$_\infty$-algebra, where
\begin{align*}
\pi_k ([r]; x_1, \ldots, x_k ) := \mu_k (Tx_1, \ldots, Tx_{r-1}, x_r, Tx_{r+1}, \ldots, Tx_k), \text{ for } k \geq 1 \text{ and } 1 \leq r \leq k.
\end{align*}
\end{thm}

\begin{proof}
Since $(\mathfrak{g}, \{ \mu_k \}_{k \geq 1})$ is a Leibniz$_\infty$-algebra, the identity (\ref{leib-inf-iden}) holds for $n \geq 1$ and any $x_1, \ldots, x_n \in \mathfrak{g}$. Let us replace the tuple $(x_1, \ldots, x_n)$ of elements in $\mathfrak{g}$ by $(Tx_1, \ldots, x_r, \ldots, Tx_n)$, for some $1 \leq r \leq n$. For fixed $i, j, \lambda$ and $\sigma \in Sh (\lambda -1, j -1)$, if $r \in \{ \sigma (1), \ldots, \sigma (\lambda -1) \}$, then the term written in the second line of (\ref{leib-inf-iden}) looks as
\begin{align*}
&\mu_i \big(  T(x_{\sigma (1)}), \ldots, x_r, \ldots, T (x_{\sigma (\lambda -1)}), \mu_j \big( T (x_{\sigma (\lambda)}), \ldots, T (x_{\sigma (\lambda + j -2)}), T (x_{\lambda + j -1})  \big), \ldots, T (x_n) \big) \\
&= \pi_i \big( R^\sigma_{i;\lambda, j} [r]; x_{\sigma (1)}, \ldots, x_{\sigma (\lambda -1)}, \pi_j \big( S^\sigma_{i;\lambda, j} [r]; x_{\sigma (\lambda)}, \ldots, x_{\sigma (\lambda + j -2)}, x_{\lambda + i -1}  \big), \ldots, x_n   \big).
\end{align*}
Similarly, if $r \in \{ \sigma (\lambda), \ldots, \sigma (\lambda + j-2), \lambda + j -1  \}$ or $r \in \{ \lambda + j , \ldots, n \}$, then the term written in the second line of (\ref{leib-inf-iden}) will have the same expression as above. Therefore, in any choice of $r$, we get the identities of a pre-Leibniz$_\infty$-algebra.
\end{proof}

\medskip

\noindent {\bf Skeletal and strict pre-Leibniz$_\infty$-algebras and their classifications.} Here we mainly concentrate on pre-Leibniz$_\infty$-algebras whose underlying graded vector space is concentrated in degrees $-1$ and $0$. We call them $2$-term pre-Leibniz$_\infty$-algebras. Two particular classes of $2$-term pre-Leibniz$_\infty$-algebras are `skeletal' and `strict' algebras. Finally, we classify skeletal and strict algebras.

Let $(\mathfrak{a} = \oplus_{n \in \mathbb{Z}} \mathfrak{a}_n, \{ \pi_k \}_{k \geq 1})$ be a pre-Leibniz$_\infty$-algebra such that $\mathfrak{a}_n = 0$ for $n \neq -1, 0$. Then it follows from the degree reason that $\pi_k = 0$ for $k \geq 4$. We write $d = \pi_1 : \mathfrak{a}_{-1} \rightarrow \mathfrak{a}_0$. Then the identities of pre-Leibniz$_\infty$-algebra can be described by the following.

\begin{defn}\label{defn-2t}
A $2$-term pre-Leibniz$_\infty$-algebra is a triple $(\mathfrak{a}_{-1} \xrightarrow{d} \mathfrak{a}_0, \pi_2, \pi_3)$ consisting of a $2$-term chain complex $\mathfrak{a}_{-1} \xrightarrow{d} \mathfrak{a}_0$, maps $\pi_2 \in \mathrm{Hom} ({\bf k}[C_2 ] \otimes \mathfrak{a}_i \otimes \mathfrak{a}_j, \mathfrak{a}_{i+j})$ for $-1 \leq i, j, i+ j \leq 0$, and a map $\pi_3 \in \mathrm{Hom} ({\bf k}[C_3] \otimes \mathfrak{a}_0^{\otimes 3}, \mathfrak{a}_{-1})$ satisfying the following set of identities
\begin{itemize}
\item[(i)] $ d \pi_2 ([r]; x, u) = \pi_2 ([r]; x, du),$\\
$d \pi_2 ([r]; u, x) = \pi_2 ([r]; du, x),$
\item[(ii)] $\pi_2 ([r]; du, v) = \pi_2 ([r]; u, dv),$
\item[(iii)] $(\pi_2 \diamond_1 \pi_2 - \pi_2 \diamond_2 \pi_2)([s];x,y,z) = d \pi_3 ([s];x,y,z),$
\item[(iv)] $(\pi_2 \diamond_1 \pi_2 - \pi_2 \diamond_2 \pi_2)([s];x,y,u) = \pi_3 ([s];x,y,du),$\\
$(\pi_2 \diamond_1 \pi_2 - \pi_2 \diamond_2 \pi_2)([s];x,u,y) =  \pi_3 ([s];x,du,y),$ \\
$(\pi_2 \diamond_1 \pi_2 - \pi_2 \diamond_2 \pi_2)([s];u,x,y) =  \pi_3 ([s];du,x,y),$
\item[(v)] $(\pi_3 \diamond_1 \pi_2 - \pi_3 \diamond_2 \pi_2 + \pi_3 \diamond_3 \pi_2) ([t];x,y,z,w) = (\pi_2 \diamond_1 \pi_3 + \pi_2 \diamond_2 \pi_3) ([t];x,y,z,w),$
\end{itemize}
for all $x, y, z, w \in \mathfrak{a}_0$, $u, v \in \mathfrak{a}_{-1}$, $[r] \in C_2$, $[s] \in C_3$ and $[t] \in C_4.$
\end{defn}

\begin{defn}
A $2$-term pre-Leibniz$_\infty$-algebra $(\mathfrak{a}_{-1} \xrightarrow{d} \mathfrak{a}_0, \pi_2, \pi_3)$ is called 
\begin{itemize}
\item `skeletal' pre-Leibniz$_\infty$-algebra if $d=0$,
\item `strict' pre-Leibniz$_\infty$-algebra if $\pi_3 = 0$.
\end{itemize}
\end{defn}


In the following result, we classify skeletal pre-Leibniz$_\infty$-algebras in terms of third cocycles of pre-Leibniz algebras.

\begin{thm}\label{thm-skeletal}
There is a one-to-one correspondence between skeletal pre-Leibniz$_\infty$-algebras and tuples $(\mathfrak{a}, M, \theta)$ consisting of a pre-Leibniz algebra $\mathfrak{a}$, a representation $M$ and a $3$-cocycle $\theta \in Z^3_\mathsf{pL} (\mathfrak{a}, M)$ of the pre-Leibniz algebra $\mathfrak{a}$ with coefficients in $M$.
\end{thm}

\begin{proof}
Let $(\mathfrak{a}_{-1} \xrightarrow{0} \mathfrak{a}_0, \pi_2, \pi_3)$ be a skeletal pre-Leibniz$_\infty$-algebra. Then it follows from Definition \ref{defn-2t} (iii) that $\mathfrak{a}_0$  is a pre-Leibniz algebra with structure maps
\begin{align*}
x \triangleleft y = \pi_2 ([1]; x, y) ~~~~ \text{ and } ~~~~ x\triangleright y = \pi_2 ([2]; x, y), ~ \text{ for } x, y \in \mathfrak{a}_0.
\end{align*}
Moreover, the condition (iv) of Definition \ref{defn-2t} implies that $(\mathfrak{a}_{-1}, \triangleleft^L, \triangleright^L, \triangleleft^R, \triangleright^R)$ is a representation of the pre-Leibniz algebra $(\mathfrak{a}_0, \triangleleft, \triangleright)$, where
\begin{align*}
x \triangleleft^L u = \pi_2 ([1];x, u), ~~~~~~ x \triangleright^L u = \pi_2 ([2]; x, u), ~~~~~~ u \triangleleft^R x = \pi_2 ([1]; u, x) ~~~~ \text{ and } ~~~~ u \triangleright^R x = \pi_2 ([2]; u, x),
\end{align*}
for $x \in \mathfrak{a}_0$, $u \in \mathfrak{a}_{-1}$. Finally, the condition (v) of Definition \ref{defn-2t} is equivalent to $(\delta_\mathsf{pL} (\pi_3)) ([t];x, y, z, w) = 0$, where $\delta_\mathsf{pL}$ is the differential on the cochain complex of the pre-Leibniz algebra $\mathfrak{a}_0$ with coefficients in the representation $\mathfrak{a}_{-1}$. Hence we obtain a required tuple $(\mathfrak{a}_0, \mathfrak{a}_{-1}, \pi_3)$.

\medskip

Conversely, let $(\mathfrak{a}, M, \theta)$ be a triple consisting of a pre-Leibniz algebra $\mathfrak{a} = (\mathfrak{a}, \triangleleft, \triangleright)$, a representation $M = (M, \triangleleft^L, \triangleright^L, \triangleleft^R, \triangleright^R)$ and a $3$-cocycle $\theta \in Z^3_\mathsf{pL}(\mathfrak{a}, M)$. Then it can be easily checked that $(M \xrightarrow{0} \mathfrak{a}, \pi_2, \pi_3)$ is a skeletal pre-Leibniz$_\infty$-algebra, where
\begin{align*}
\pi_2 ([1]; x, y) =& x \triangleleft y, \quad \pi_2 ([2]; x, y) = x \triangleright y, \quad \pi_2 ([1]; x, u) = x \triangleleft^L u, \quad \pi_2 ([2]; x, u) = x \triangleright^L u, \\
&\pi_2 ([1]; u, x) = u \triangleleft^R x, \quad \pi_2 ([2]; u, x) = u \triangleright^R x ~~~~~ \text{ and } ~~~~ \pi_3 = \theta.
\end{align*}
Finally, the above two correspondences are inverses to each other. This completes the proof.
\end{proof}

In the following, we introduce crossed module of pre-Leibniz algebras and using them, we classify strict pre-Leibniz$_\infty$-algebras.

\begin{defn}
A crossed module of pre-Leibniz algebras is a tuple $(\mathfrak{a}, \mathfrak{b}, d, \pi^L, \pi^R)$ consisting of two pre-Leibniz algebras $\mathfrak{a}$ and $\mathfrak{b}$, a pre-Leibniz algebra morphism $d : \mathfrak{a} \rightarrow \mathfrak{b}$ and maps
\begin{align*}
\pi^L \in \mathrm{Hom} (\mathbf{k}[C_2] \otimes \mathfrak{b} \otimes \mathfrak{a}, \mathfrak{a}), \qquad \quad \pi^R \in \mathrm{Hom} (\mathbf{k}[C_2] \otimes \mathfrak{a} \otimes \mathfrak{b}, \mathfrak{a})
\end{align*}
which makes $\mathfrak{a}$ into a representation of the pre-Leibniz algebra $\mathfrak{b}$, satisfying additionally for $x, y \in \mathfrak{a}$, $b \in \mathfrak{b}$, $[r] \in C_2$ and $[s] \in C_3$,
\begin{itemize}
\item[(i)] $d ( \pi^L ([r];b,x)) = \pi_\mathfrak{b} ([r]; b, d(x)),$ \\
$d (\pi^R ([r]; x, b)) = \pi_\mathfrak{b} ([r]; d(x), b),$
\item[(ii)] $\pi^L ([r]; d(x), y)= \pi_\mathfrak{a} ([r]; x, y),$\\
$\pi^R ([r]; x, d(y)) = \pi_\mathfrak{a} ([r]; x, y),$
\item[(iii)] $(\pi^R \diamond_1 \pi_\mathfrak{a}) ([s];x,y, b) = (\pi_\mathfrak{a} \diamond_2 \pi^R) ([s]; x, y, b),$\\
$(\pi_\mathfrak{a} \diamond_1 \pi^R) ([s]; x,b,y) = (\pi_\mathfrak{a} \diamond_2 \pi^L) ([s];x, b, y),$\\
$(\pi_\mathfrak{a} \diamond_1 \pi^L)([s];b,x,y) = (\pi^L \diamond_2 \pi_\mathfrak{a}) ([s];b,x,y).$
\end{itemize}
\end{defn}

\begin{thm}\label{thm-strict}
There is a one-to-one correspondence between strict pre-Leibniz$_\infty$-algebras and crossed module of pre-Leibniz algebras.
\end{thm}

\begin{proof}
Let $(\mathfrak{a}_{-1} \xrightarrow{d} \mathfrak{a}_0, \pi_2, \pi_3 = 0)$ be a strict pre-Leibniz$_\infty$-algebra. Define an element 
\begin{align*}
\pi_{\mathfrak{a}_{-1}} \in \mathrm{Hom} ({\bf k} [C_2] \otimes (\mathfrak{a}_{-1})^{\otimes 2}, \mathfrak{a}_{-1}) ~~\text{ by } ~~ \pi_{\mathfrak{a}_{-1}} ([r]; u, v) := \pi_2 ([r]; du, v) = \pi_2 ([r]; u, dv),
\end{align*}
for $[r] \in C_2$ and $u, v \in \mathfrak{a}_{-1}$. Then $\pi_{\mathfrak{a}_{-1}}$ defines a pre-Leibniz algebra structure on $\mathfrak{a}_{-1}$. We also define an element 
\begin{align*}
\pi_{\mathfrak{a}_{0}} \in \mathrm{Hom} ({\bf k} [C_2] \otimes (\mathfrak{a}_{0})^{\otimes 2}, \mathfrak{a}_{0}) ~~\text{ by } ~~ \pi_{\mathfrak{a}_{0}} ([r]; x,y) := \pi_2 ([r]; x,y), \text{ for } [r] \in C_2 \text{ and } x, y\in \mathfrak{a}_0.
\end{align*}
It follows from Definition \ref{defn-2t} (iii) that $\pi_{\mathfrak{a}_0}$ defines a pre-Leibniz algebra structure on $\mathfrak{a}_0$. Moreover, the map $d : \mathfrak{a}_{-1} \rightarrow \mathfrak{a}_0$ is a pre-Leibniz algebra morphism as
\begin{align*}
d \big(  \pi_{\mathfrak{a}_{-1}} ([r]; u, v) \big) = d (\pi_2 ([r]; du, v)) = \pi_2 ([r]; du, dv) = \pi_{\mathfrak{a}_0} ([r]; du, dv).
\end{align*}
Finally, we define maps $\pi^L \in \mathrm{Hom} ({\bf k} [C_2] \otimes \mathfrak{a}_0 \otimes \mathfrak{a}_{-1}, \mathfrak{a}_{-1})$ and $\pi^R \in \mathrm{Hom} ({\bf k} [C_2] \otimes \mathfrak{a}_{-1} \otimes \mathfrak{a}_{0}, \mathfrak{a}_{-1})$ by
\begin{align*}
\pi^L ([r];x,u) = \pi_2 ([r];x,u) ~~ \text{ and } ~~ \pi^R ([r];u,x) = \pi_2 ([r];u,x),~ \text{ for } x \in \mathfrak{a}_0, u \in \mathfrak{a}_{-1}.
\end{align*}
It is easy to see that $(\mathfrak{a}_{-1}, \mathfrak{a}_0, d, \pi^L, \pi^R)$ is a crossed module of pre-Leibniz algebras.

Conversely, let $(\mathfrak{a}, \mathfrak{b}, d, \pi^L, \pi^R)$ be a crossed module of pre-Leibniz algebras. Then it is straightforward to verify that $(\mathfrak{a} \xrightarrow{d} \mathfrak{b}, \pi_2 , \pi_3  = 0)$ is a strict pre-Leibniz$_\infty$-algebra, where
\begin{align*}
\pi_2 ([r]; b, b') = \pi_\mathfrak{b} ([r];b, b'), ~~~~ \pi_2 ([r];b, x) = \pi^L ([r]; b, x) ~~ \text{ and } ~~ \pi_2 ([r];x, b) = \pi^R ([r];x,b),
\end{align*}
for $b, b' \in \mathfrak{b}$  and $ x \in \mathfrak{a}.$ This completes the proof.
\end{proof}

\begin{exam}
Let $\mathfrak{a}$ be a pre-Leibniz algebra, and let $\pi \in \mathrm{Hom}({\bf k}[C_2]\otimes \mathfrak{a}^{\otimes 2}, \mathfrak{a})$ be the corresponding Maurer-Cartan element. Then $(\mathfrak{a} \xrightarrow{\mathrm{id}} \mathfrak{a}, \pi_2 = \pi, \pi_3 = 0)$ is a strict pre-Leibniz$_\infty$-algebra. The corresponding crossed module of pre-Leibniz algebras is given by $(\mathfrak{a}, \mathfrak{a}, \mathrm{id}, \pi, \pi).$
\end{exam}

\medskip



\noindent {\bf Acknowledgements.}  The author would like to thank Indian Institute of Technology (IIT) Kharagpur for providing the beautiful academic environment where the research has been carried out.


\end{document}